\documentclass[11pt]{amsart}

\usepackage{mathtools}
\newlength{\myhmargin}  
\usepackage[textheight=574pt, textwidth=445pt, marginparwidth=50pt, centering]{geometry}

\usepackage{amsmath,amssymb,amsthm,amsfonts,amscd,flafter,
graphicx,verbatim,pinlabel,mathrsfs,caption, enumitem,dsfont}
\usepackage[all]{xy}
\usepackage{epstopdf}
\usepackage{verbatim}
\usepackage[colorlinks=false]{hyperref}
\usepackage[all]{hypcap}
\usepackage{xcolor}
\usepackage{url}
\AtBeginDocument{\addtocontents{toc}{\protect\setlength{\parskip}{0pt}}}

\newcommand{\longcomment}[2]{#2}

\DeclareFontFamily{U}{mathx}{\hyphenchar\font45}
\DeclareFontShape{U}{mathx}{m}{n}{
      <5> <6> <7> <8> <9> <10>
      <10.95> <12> <14.4> <17.28> <20.74> <24.88>
      mathx10
      }{}
\DeclareSymbolFont{mathx}{U}{mathx}{m}{n}
\DeclareFontSubstitution{U}{mathx}{m}{n}
\DeclareMathAccent{\widecheck}{0}{mathx}{"71}
\newcommand{\HMto}{\widecheck{\mathit{HM}}}

\longcomment{
    \RequirePackage{rotating}                   
    \def\HMto{%
       \setbox0=\hbox{$\widehat{\mathit{HM}}$}
       \setbox1=\hbox{$\mathit{HM}$}
       \dimen0=1.1\ht0
       \advance\dimen0 by 1.17\ht1
       \smash{\mskip2mu\raise\dimen0\rlap{%
          \begin{turn}{180}
              {$\widehat{\phantom{\mathit{HM}}}$}
           \end{turn}} \mskip-2mu    
                \mathit{HM}
    }{\vphantom{\widehat{\mathit{HM}}}}{}}
}
    
    \newcommand*\oline[1]{%
  \vbox{%
    \hrule height 0.35pt
    \kern0.1ex
    \hbox{%
      \kern-0.0em
      \ifmmode#1\else\ensuremath{#1}\fi
      \kern-0.1em
    }
  }
}

\newtheorem{theorem}{Theorem}[section]
\newtheorem{lemma}[theorem]{Lemma}

\newtheorem{corollary}[theorem]{Corollary}

\newtheorem{proposition}[theorem]{Proposition}
\newtheorem{question}[theorem]{Question}

\theoremstyle{definition}
\newtheorem{definition}[theorem]{Definition}

\newtheorem{remark}[theorem]{Remark}

\makeatletter
\newtheorem*{rep@thm}{\rep@title}
\newcommand{\newreptheorem}[2]{%
\newenvironment{rep#1}[1][0,0]{%
\def\rep@title{#2##1}%
\begin{rep@thm}}%
{\end{rep@thm}}}
\makeatother
\newreptheorem{theorem}{}

\title{Quasipositive surfaces and decomposable Lagrangians}

\author[Lev Tovstopyat-Nelip]{Lev Tovstopyat-Nelip}
\address{Department of Mathematics \\ Michigan State University}
\email{tovstopy@msu.edu}

\begin{document}

\begin{abstract} 
We show that a quasipositive surface with disconnected boundary induces a map between the knot Floer homology groups of its boundary components preserving the transverse invariant. As an application, we show that this invariant can be used to obstruct decomposable Lagrangian cobordisms of arbitrary genus within Weinstein cobordisms. The construction of our maps rely on the comultiplicativity of the transverse invariant. 
Along the way, we also recover various naturality statements for the invariant under contact +1 surgery. \end{abstract}

\maketitle


\vspace{-1cm}
\section{introduction}
\label{sec:intro}

Various authors have defined effective invariants of Legendrian and transverse links using the knot Floer homology package \cite{grid, stipV, LOSS}. The most general of these is the so-called BRAID invariant \cite{equiv}, which associates to a transverse link $K\subset (Y,\xi)$ classes
\[
t^\circ(K)\subset HFK^\circ (-Y,K), \text{ where $\circ\in \{\wedge,-\}$}.
\]
Via Legendrian approximation, we think of the BRAID invariant as an invariant of Legendrian links which is preserved under negative stabilizations.
Our main result is that the BRAID invariant can be used to obstruct decomposable Lagrangian link cobordisms within Weinstein cobordisms\footnote{decomposable Lagrangians in the symplectization of $(S^3,\xi_{std})$ are built up from elementary cobordisms associated to \emph{births} and \emph{pinches}, we extend this notion to the setting of Weinstein cobordisms in Section \ref{sec:lagr}.} between closed contact manifolds, with no restrictions on the genus:


\begin{theorem}
\label{lagrangianknot}
Let $(W,\omega): (Y_1,\xi_1)\to (Y_2,\xi_2)$ be a Weinstein cobordism. A decomposable Lagrangian cobordism $S\subset W$ between Legendrian knots $L_i\subset (Y_i,\xi_i)$ induces an $\mathbb{F}[U]$-module homomorphism 
\[
f_S: HFK^{-}(-Y_2,L_2)\to HFK^{-}(-Y_1,L_1)
\]
which is natural with respect to the Legendrian invariant, $t^-(L_2)\xrightarrow{f_S} t^-(L_1)$. Moreover, we have a relationship between the $\widehat{t}$-invariants, if $\widehat{t}(L_1)\ne 0$ then $\widehat{t}(L_2)\ne 0$. 
\end{theorem}

In this paper all links are oriented; given a link $K$ we let $-K$ denote the same link equipped with reversed orientation. Reversing the orientation of a Lagrangian cobordism from $L_1$ to $L_2$ produces a Lagrangian cobordism from $-L_1$ to $-L_2$. Theorem \ref{lagrangianknot} has the following corollary:

\begin{corollary}
\label{uaction}
Suppose $L_i\subset (Y_i,\xi_i)$ are Legendrian knots such that, for some $n$, either
\begin{itemize}
\item $t^- (L_1)\notin Im(U^n) \text{ and } t^-(L_2)\in Im(U^n),$ \text{or}
\item $t^- (-L_1)\notin Im(U^n) \text{ and } t^-(-L_2)\in Im(U^n)$
\end{itemize}
then there is no Lagrangian cobordism from $L_1$ to $L_2$ inside any Weinstein cobordism from $(Y_1,\xi_1)$ to $(Y_2,\xi_2)$.
\end{corollary}

Reversing the orientation of a Legendrian link exchanges the operations of positive and negative Legendrian stabilizations. As mentioned above, the invariant is unchanged under negative stabilizations, on the other hand positive stabilization corresponds to multiplying $t^-$ by the formal variable $U$ and killing the invariant $\widehat{t}$. Thus the above corollary obstructs a Lagrangian cobordism from $L_1$ to $L_2$ in the case that either
\begin{itemize}
\item $\widehat{t}(L_1)\ne 0$ \text{ and } $L_2$ \text{ admits a positive destabilization, or}
\item $\widehat{t}(-L_1)\ne 0$ \text{ and } $L_2$ \text{ admits a negative destabilization.}
\end{itemize}

Theorem \ref{lagrangianknot} follows from a more general statement for links, Theorem \ref{decomp}, which can be used to obstruct decomposable Lagrangian fillings of links:
\begin{corollary}
Suppose $L\subset (Y,\xi)$ is a Legendrian link.  If either $\widehat{t}(L)=0$ or $\widehat{t}(-L)=0$, then there is no decomposable Lagrangian filling of $L$ in any Stein filling of $(Y,\xi)$.
\end{corollary}

Our results fit into an existing body of work showing that knot Floer homology effectively obstructs Lagrangian link cobordisms. For the hat theory, in the case that the Weinstein cobordism is the symplectization of the tight 3-sphere, the above results have already been established by Baldwin, Lidman and Wong \cite{Lagr}. They work with the so-called GRID invariants defined in \cite{grid}; their arguments utilize grid diagrams, and are specific to the setting they work in. 
There are also various results for obstructing Lagrangian concordances, i.e. genus zero cobordisms, equipped with various regularity conditions.
For exact Lagrangians in symplectizations of contact manifolds see the works of Baldwin and Sivek \cite{BS1, BS2}, and for regular Lagrangians in Weinstein cobordisms see the paper by Golla and Juh\'asz \cite{func}. 

\begin{remark}
In the case that a Weinstein cobordism is the symplectization of the tight 3-sphere, it is unclear whether our definition of decomposable Lagrangian is more general than the one previously studied in the literature. Our definition allows one to consider any Weinstein structure, as opposed to just the trivial one. In particular, Kyle Hayden pointed out to the author that a regular Lagrangian cobordism may be decomposable with respect to a \emph{special Weinstein structure}, as defined in \cite{flexible}.
\end{remark}

To establish Theorem \ref{lagrangianknot} and its generalization, Theorem \ref{decomp}, we prove the result for certain elementary cobordisms associated to \emph{births}, \emph{pinches} and Weinstein handle attachments. The most interesting case is that of the pinch-move. If $L_1$ and $L_2$ differ by a pinch move \footnote{in particular $L_1$ and $L_2$ can be thought of as lying in the same contact manifold.}, we show that their transverse push-offs cobound some quasipositive surface realized as a subsurface of a page of an open book supporting the ambient contact structure. We then appeal to:


\begin{theorem}
\label{qpm}
Let $(B,\pi)$ be an open book supporting $(Y,\xi)$, and $S\subset \Sigma = \pi^{-1}(0)$ a subsurface whose oriented boundary is the union of knots $K_1$ and $K_2$. If $K_1$ is boundary compatible,
there exists an $\mathbb{F}[U]$-module homomorphism
\[
f_S: HFK^{-}(-Y, K_1)\to HFK^{-}(-Y, -K_2)
\]
which is natural with respect to the transverse invariant, $t^-(K_1) \xrightarrow{f_S} t^-(-K_2)$. Moreover, we have a relationship between the $\widehat{t}$-invariants, if $\widehat{t}(-K_2)\ne 0$ then $\widehat{t}(K_1)\ne 0$. 
\end{theorem}
See Theorem \ref{qpm1} for a more general statement for links. See Section \ref{boundarycomp} for the definition of boundary compatible; this is a less restrictive notion than non-isolating. Given an isotopy class of a boundary compatible link in a page of an open book, we associate to it a well defined transverse isotopy class. In the case that the link is non-isolating, we prove that our associated transverse link is the push-off of the unique Legendrian isotopy class given by Honda's Legendrian realization principle \cite{honda, LOSS}, see Theorem \ref{uniqueleg}.

\begin{theorem}
\label{uniquetransapprox}
Suppose $(B,\pi)$ is an open book supporting a contact manifold $(Y,\xi)$, and $L$ a boundary compatible link in a page $\Sigma = \pi^{-1}(0)$. The smooth isotopy class of $L$ in $\Sigma$ uniquely determines a transverse link $T(L)$ in $(Y,\xi)$ up to transverse isotopy. Moreover, if $L$ is non-isolating, and hence naturally Legendrian, then $L_+ = T(L)$.
\end{theorem}

Our construction involves braiding the boundary compatible link about the open book using push-maps, making it transverse in the process. Instead of braiding, one could construct a transverse link using Giroux flexiblity. These two approaches produce the same transverse isotopy class, see Remark \ref{flexible}.
Theorem \ref{qpm1} is established via a detailed study of pointed monodromies, and applying comultiplicativity of the BRAID invariant \cite{braiddynamics}.
\begin{question}
There are choices involved in the construction of the map $f_S$ of Theorems \ref{qpm} and \ref{qpm1}. For any set of choices, the resulting map preserves the BRAID invariant. Is the map independent of the choices made? In particular, is this the map associated to pushing the surface into the Weinstein cobordism and equipping it with some decorations, as in \cite{juhasz}?

\end{question}


Along the way to establishing Theorem \ref{lagrangianknot}, we recover and generalize various naturality statements involving the BRAID invariant and contact +1 surgeries in Section \ref{sec:contactsurg}. We also obtain a non-vanishing result for the transverse invariant $\widehat{t}$ of a strongly-quasipositive link in a contact manifold with non-vanishing Ozsv\'ath-Szab\'o contact invariant, see Corollary \ref{corsqp}.

Of independent interest, we make some observations about ribbons of Legendrian graphs which allow us to generalize a classical result due to Lyon \cite{Lyon}; the proof given is little more than a translation of Lyon's original argument into modern contact geometric language.

\begin{theorem}
\label{thmlyon}
Let $S\subset Y$ be a Seifert surface for a link $K$. Then there exists a contact structure on $Y$ such that $S$ is the ribbon of some Legendrian graph. In particular $S$ can be realized as a $\pi_1$-injective subsurface of a page of an open book decomposition for $Y$.
\end{theorem}

\begin{corollary}
Let $S\subset Y$ be a Seifert surface for a link $K$. One may attach bands to $S$ to obtain the page of some open book for $Y$.
\end{corollary}

We introduce the following notion, which measures how far a null-homologous link in a 3-manifold is from being fibered, and pose some questions:
\begin{definition}
Given a null-homologous link $K\subset Y$, we define \emph{fiber depth} of $K$ to be the number of bands one must attach to a Seifert surface for $K$ to obtain a page of an open book for $Y$, minimized over all Seifert surfaces for $K$ and open books for $Y$.
\end{definition}

\begin{question}
Is there a relationship between fiber depth and Morse-Novikov number? What is the relationship between the fiber depth and the depth of a knot, where the latter is the notion studied by Gabai \cite{gabai}?

\end{question}

\begin{question}
The knot Floer homology of a link in the 3-sphere in an extremal Alexander grading is rank 1 if and only if the link is fibered \cite{linkgenus}, which is equivalent to the fiber depth being zero. 
Is there a more general relationship between the rank of knot Floer homology in an extremal Alexander grading and the fiber depth? Juh\'asz \cite{juha} has established some relationship between depth and the rank of this summand.
\end{question}


\subsection{Acknowledgements} We thank John Baldwin, Kyle Hayden, Matthew Hedden and Marc Kegel for some interesting conversations. 

\section{Contact geometric background}
\label{sec:background}

\subsection{Legendrian Ribbons and quasipositive surfaces}

We collect various notions and results about graphs and surfaces in contact 3-manifolds which are well known to experts, although we could not find a reference for Theorem \ref{thmlyon} in the literature.

Open book decompositions were first studied from the perspective of contact geometry by Thurston and Winkelnkemper \cite{thurwink}, who introduced the following notion:
\begin{definition}
An open book $(B,\pi)$ \emph{supports} a contact manifold $(Y,\xi)$ if there exists a 1-form $\alpha$ s.t. $ker(\alpha) = \xi$ such that
\begin{itemize}
\item $\alpha|_B >0$, i.e. $\alpha$ specifies the orientation of $B$,
\item $d\alpha$ restricts to an area form on every page of the open book.
\end{itemize}
\end{definition}
Away from a neighborhood of the binding, the contact planes can be made arbitrarily close to the tangencies of the pages of the open book. \cite{thurwink} proved that every contact structure on a closed 3-manifold admits a supporting open book.

Giroux \cite{giroux} introduced the following notions:

\begin{definition}
A \emph{contact cell decomposition} of $(Y,\xi)$ is a CW-decomposition of $Y$ satisfying
\begin{itemize}
\item the 1-skeleton is a Legendrian graph,
\item each 2-cell is convex, attached along a $tb= -1$ unknot, and
\item each 3-cell is a Darboux ball.
\end{itemize}
\end{definition}

\begin{definition}
Let $G\subset (Y,\xi)$ be a Legendrian graph. A Ribbon $R(G)$ of $G$ is a smoothly embedded surface in $Y$ satisfying
\begin{itemize}
\item
$G$ is contained in $R(G)$,
\item
there is a contact form $\alpha$ whose Reeb vector field is positively transverse to $R(G)$, and
\item
there exists a vector field on $R(G)$ 
directing the characteristic foliation, positively transverse to $\partial R(G)$, whose inverse flow specifies a retraction of $R(G)$ onto $G$.
\end{itemize}
\end{definition}

Every contact 3-manifold admits a contact cell decomposition.
Giroux proved that a ribbon of the 1-skeleton of any contact cell decomposition is a page of an open book supporting the underlying contact structure. 

The following definition and result are due to Honda \cite{contact1}:
\begin{definition}
Let $\Sigma$ denote a connected surface with boundary. A graph $G$ in the interior of the surface is said to be \emph{non-isolating in $\Sigma$} if every component of $\Sigma\smallsetminus G$ meets $\partial \Sigma$.
\end{definition}

\begin{lemma}
\label{LRP}
If $G$ is a non-isolating graph in the page of an open book $\pi^{-1}(0)=\Sigma$, 
then $\Sigma$ can be isotoped to make $G$ Legendrian so that a neighborhood of $G$ in $\Sigma$ is a ribbon $R(G)$.
\end{lemma}
\begin{proof}
The Heegaard surface $F$ obtained by gluing two pages of an open book together can be made convex, with dividing set the binding \cite{torisu}. We identify the page $\Sigma\subset F$ as the positive region. The Legendrian Realization Principle then says that $F$ can be isotoped through convex surfaces to make $G\subset \Sigma \subset F$ Legendrian. Moreover, the contact vector field of $F$ is positively transverse to both $TF$ and $\xi$ along $G$, hence the contact and surface framings of $G$ agree.
\end{proof}

For a non-isolating link $K\subset \Sigma$, the produced Legendrian representative is unique up to Legendrian isotopy. This is stated and proven for Legendrian knots in Section 2.2 of \cite{LOSS}, but their argument generalizes to the case of links in a straightforward manner.

\begin{theorem}
\label{uniqueleg}
Let $(B,\pi)$ be an open book supporting $(Y,\xi)$, and
$K$ be a non-isolating link in a page of the open book $\pi^{-1}(0)=\Sigma$. The smooth isotopy class of $K$ in $\Sigma$ uniquely determines a Legendrian link in $(Y,\xi)$ up to Legendrian isotopy.
\end{theorem}

Following Rudolph \cite{qprudolph}, we introduce the following notions:
\begin{definition}
A Seifert surface $S$ for a link $K\subset (Y,\xi)$ is \emph{quasipositive} if it is ambient isotopic to a $\pi_1$-injective subsurface of a page of an open book supporting $\xi$. 
$K$ is said to be \emph{strongly quasipositive} if it admits a quasipositive Seifert surface.
\end{definition}

\begin{theorem}
\label{thmqpgraph}
A Seifert surface $S \subset (Y,\xi)$ is quasipositive if and only if it can be realized as a ribbon of some Legendrian graph.
\end{theorem}
\begin{proof}
Suppose $S$ is a quasipositive surface. By definition, we can assume that $S$ is a $\pi_1$-injective subsurface of $\Sigma$, a page of an open book supporting $\xi$. Choose a graph $G\subset S$ which $S$ retracts onto. By construction, $G$ is not isolating in $\Sigma$ and made Legendrian by Lemma \ref{LRP}, so that the contact and surface framings on $G$ agree.  Any ribbon $R(G)$ is then ambiently isotopic to $S$.

For the reverse direction, begin with a Legendrian graph $G$. Include $G$ into the 1-skeleton $\Gamma$ of a contact cell decomposition. By construction, a ribbon $R(G)$ is then contained as a $\pi_1$-injective subsurface of $R(\Gamma)$ - a page of an open book supporting the contact structure. Thus, any ribbon of a Legendrian graph is a quasipositive surface.
\end{proof}

As observed in \cite{BCV}, every Seifert surface $S\subset Y$ is the ribbon of a Legendrian graph in \emph{some} contact structure on $Y$, this allows us to generalize a classical result of \cite{Lyon}.
\begin{proof}[Proof of Theorem \ref{thmlyon}]
Fix a contact structure $\xi$ on $Y$. Choose a graph $G\subset S$ which the surface retracts onto. Let $\Gamma$ denote a Legendrian realization of $G$ within $(Y,\xi)$. 
A decomposition of $\Gamma$ into vertices and edges induces an obvious handle decomposition of a ribbon $R(\Gamma)$; we refer to the 1-handles of the ribbon as \emph{bands}.
$R(\Gamma)$ differs from $S$ by some number of positive or negative twists along the bands. To add a negative twist along a band of $R(\Gamma)$, we apply a Legendrian stabilization to the corresponding edge of $\Gamma$. To add a positive twist, we change the contact structure to an overtwisted one. By taking the contact connected sum of $Y$ with an overtwisted 3-sphere, and Legendrian connected summing an edge of $\Gamma$ with a $tb=0$ unknot, we add a positive twist to the corresponding band.

The latter statement follows by Theorem \ref{thmqpgraph}.

\end{proof}

\subsection{Transverse links and braids about open books}

Bennequin \cite{ben} studied transverse links in the tight 3-sphere via braids. His ideas have been fully generalized and used to study transverse links in arbitrary contact manifolds.


A link is \emph{braided} about an open book if it lies in the complement of the binding and is positively transverse to the pages; such a link intersects each page of the open book in a fixed number of points called the \emph{braid index}. 
A link braided about an open book is naturally transverse to the supported contact structure. Two such links are \emph{braid isotopic} if they are isotopic through braided links; braid isotopic links are transversely isotopic.


There is a notion of positive Markov stabilization for braids with respect to an arbitrary open book.
This operation increases the braid index by one, but preserves the transverse isotopy class of the braid. The following is a generalization of the transverse Markov theorem of Wrinkle \cite{wrinkle}, and independently Orekov and Shevchishin \cite{orevkov}:

\begin{theorem} \cite{pav}
\label{thm:markov}
Suppose $(B,\pi)$ is an open book supporting $(Y,\xi)$. Every transverse link in $(Y,\xi)$ is transversely isotopic to a braid with respect to $(B,\pi)$.
Moreover,
two links braided about an open book $(B,\pi)$ are transversely isotopic if and only if 
they are braid isotopic after some number of positive Markov stabilizations.
\end{theorem}

Braid isotopy classes of links braided about the trivial
 open book for the 3-sphere, having braid index $n$, can be identified with conjugacy classes in the Artin braid group $B_n$, which is isomorphic to the mapping class group  of a disk with marked points \[Mod(D^2\smallsetminus \{p_1,\dots,p_n\},\partial D^2).\] We think of the marked points as intersections between the braided link and a fixed page of the open book. This idea generalizes to braids about arbitrary open books.

 If $K$ is braided about $(B,\pi)$ with index $n$, then the braid isotopy class of $K$ is equivalent to a conjugacy class of an element \[g\in Mod(\Sigma\smallsetminus \{p_1,\dots,p_n\},\partial \Sigma),\] where $\Sigma = \pi^{-1}(0)$ is a page of the open book. We may take the \emph{closure} of $g$ to recover $K$. Note that the forgetful map $Mod(\Sigma\smallsetminus \{p_1,\dots,p_n\},\partial \Sigma)\to Mod(\Sigma,\partial \Sigma)$ must send $g$ to the open book monodromy.

Let $\gamma\subset \Sigma \smallsetminus \{ \{p_1,\dots,p_n\}\cup \partial \Sigma\}$ denote a properly embedded arc connecting two marked points $p_i\ne p_j$. 
We let $\sigma_\gamma$ denote the \emph{right-hand half twist about} $\gamma$; $\sigma_\gamma$ has support in a small neighborhood of $\gamma$, see Figure \ref{fig:halftwist}. 
If $\delta$ is a simple closed curve in $\Sigma \smallsetminus \{ \{p_1,\dots,p_n\}\cup \partial \Sigma\}$, we let $\tau_\delta$ denote a \emph{right-hand Dehn twist about} $\delta$. 
Mapping class groups of surfaces with marked points are generated by Dehn twists and half twists. 

We will also make use of push-maps. Let $r$ denote an oriented properly embedded arc $\gamma\subset \Sigma \smallsetminus \{ \{p_1,\dots,p_n\}\cup \partial \Sigma\}$ having endpoints at a single point $p\in P$. 
Let $\delta_1\cup \delta_2$ denote the oriented boundary of a small annular neighborhood of $r$, where the orientation of $\delta_1$ agrees with that of $r$.
The \emph{push-map} of $p$ along $r$, denoted $\pi_r$, is the composition $\tau_{\delta_1} \circ \tau_{\delta_2}^{-1}$. 
This map has a simple interpretation: the oriented arc $r$ specifies a movie of how the component of the closure of $\pi_r$ corresponding to the puncture $p$ intersects the pages of the open book as one flows through the fibration. If one forgets the basepoint $p$, then the right and left-hand Dehn twists cancel, so introduction of push-maps does not change the underlying open book.

\begin{figure}[h]
\def\svgwidth{140pt}
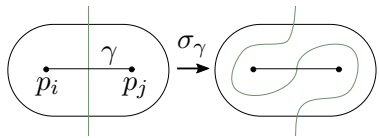
\caption{The right-hand half twist $\sigma_\gamma$ is supported in a neighborhood of the horizontal arc $\gamma$ and is determined by where it sends the vertical arc.}
\label{fig:halftwist}
\end{figure}

\section{Boundary compatible links and transverse representatives, uniqueness}
\label{boundarycomp}
\begin{definition}
\label{boundary}
An oriented link $L\subset \Sigma$ is \emph{boundary compatible} if there exist a collection of \emph{auxiliary arcs} - disjoint, properly embedded, and oriented arcs $\{\alpha_i\}\subset \Sigma\smallsetminus L$, one for each component $L_i\subset L$, satisfying 
\begin{itemize}
\item
the tangent vectors to $\partial \Sigma$ and $\alpha_i$, in that order, determine the orientation of $\Sigma$, and
\item
the tangent vectors to $L_i$ and $\alpha_i$, in that order, determine the orientation of $\Sigma$.
\end{itemize}
\end{definition}


Boundary compatibility is a strictly weaker condition than being non-isolating. In particular, boundary compatible links can be isolating, and arise naturally as boundaries of quasipositive surfaces. The main goal of this section is to prove Theorem \ref{uniquetransapprox}. 


\begin{remark}
\label{flexible}
Given a boundary compatible link $L\subset \Sigma$, we will braid the link, making it transverse in the process. Alternatively, one can apply Giroux flexibility \cite{giroux} to perturb $\Sigma$ so that the characteristic foliation, and hence $\xi$, is positively transverse to $L$. The two resulting transverse representatives can be shown to agree. Choosing the foliation so that each of the auxiliary arcs is contained in a leaf, our braiding procedure can be performed via transverse isotopy supported in a small neighborhood of $\Sigma$.
\end{remark}

Let $L$ denote a boundary compatible link in the page of an open book $\Sigma = \pi^{-1}(0)$. 
A choice of auxiliary arcs guide an isotopy of the link which braids it about the open book in the following way.
The arc $\alpha_i$ guides an isotopy of the component $L_i$, so that a sub-arc $\beta_i$ now lies in a collar neighborhood of the boundary which is fixed by the open book monodromy $\phi$, as in Figure \ref{fig:transapprox}. Let $K_i$ denote the result of this isotopy.

\begin{figure}[h]
\def\svgwidth{100pt}
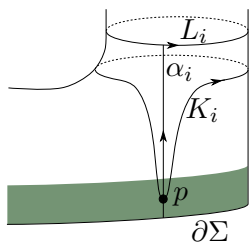
\caption{A collar neighborhood of $\partial \Sigma$ is shaded.}
\label{fig:transapprox}
\end{figure}

Let $K_i'$ denote the connect sum of $K_i$, along $\beta_i$ with a small positively braided meridian unknot for the binding of the open book. $K_i'$ is isotopic, in the complement of $\partial \Sigma$, to a knot $K''_i$ which is braided about the open book with index one.

It is easy to specify a pointed monodromy having closure $K''_i$.
Introduce a single marked point $p = \alpha_i\cap \beta_i$ and compose the open book monodromy with a push-map of $p_i$ along $K_i$, i.e. consider
$g_i = \phi\circ \pi_{K_i}$.
The closure of $g_i$ is braid isotopic to $K_i''$. Performing this procedure once for every component of $L$ realizes the entire link as a braid having index equal to the number of its components.
The construction of the push map $\pi_{K_i}$ depended on the arc $\{\alpha_i\}$ from the binding of the open book to $L_i$, so we introduce the notation $\pi _{\alpha_i}:=\pi_{K_i}$.

Suppose $L$ is boundary compatible. a choice of auxiliary arcs $\{\alpha_i\}$ gives rise to a braiding $\phi \circ \prod_i \pi_{\alpha_i}$
having transverse closure $T(L,\{\alpha_i\})$.
The resulting transverse representative of $L$ is actually independent of the choice of auxiliary arcs, up to transverse isotopy:

\begin{proposition}
\label{uniquetrans}
Suppose $(B,\pi)$ is an open book supporting a contact manifold $(Y,\xi)$, and $L$ a boundary compatible link in a page $\Sigma = \pi^{-1}(0)$. The smooth isotopy class of $L$ in $\Sigma$ uniquely determines a transverse link $T(L)$ in $(Y,\xi)$ up to transverse isotopy.
\end{proposition}

The proof of Proposition \ref{uniquetrans} relies on the following two lemmas:
\begin{lemma}
\label{disjoint}
Let $L\subset \Sigma = \pi^{-1}(0)$ be a boundary compatible knot. Suppose $\alpha,\alpha'\subset \Sigma\smallsetminus L$ are two disjoint properly embedded auxiliary arcs. Then $T(L,\alpha)$ and $T(L,\alpha')$ are transversely isotopic.
\end{lemma}
\begin{proof}

It is clear from the construction that isotoping an auxiliary arc through auxiliary arcs has no effect on the resulting push-map
as an element of the pointed mapping class group, 
and hence no effect on the transverse isotopy class of the closure.

After possibly applying such as isotopy to $\alpha$, we may assume that the arcs $\alpha$ and $\alpha'$ meet at a point $\{e\} = \alpha\cap \alpha'\in L$. Let $p$ (respectively $p'$) denote the marked point introduced in the construction of $\pi_{\alpha}$ (respectively $\pi_{\alpha'}$). Let $\gamma \subset \Sigma\smallsetminus \{p,p'\}$ be the embedded arc from $p$ to $p'$ obtained by pushing a segment of $\alpha\cup\alpha'$ off of $L$, see Figure \ref{fig:arcinvt}.

\begin{figure}[h]
\def\svgwidth{130pt}
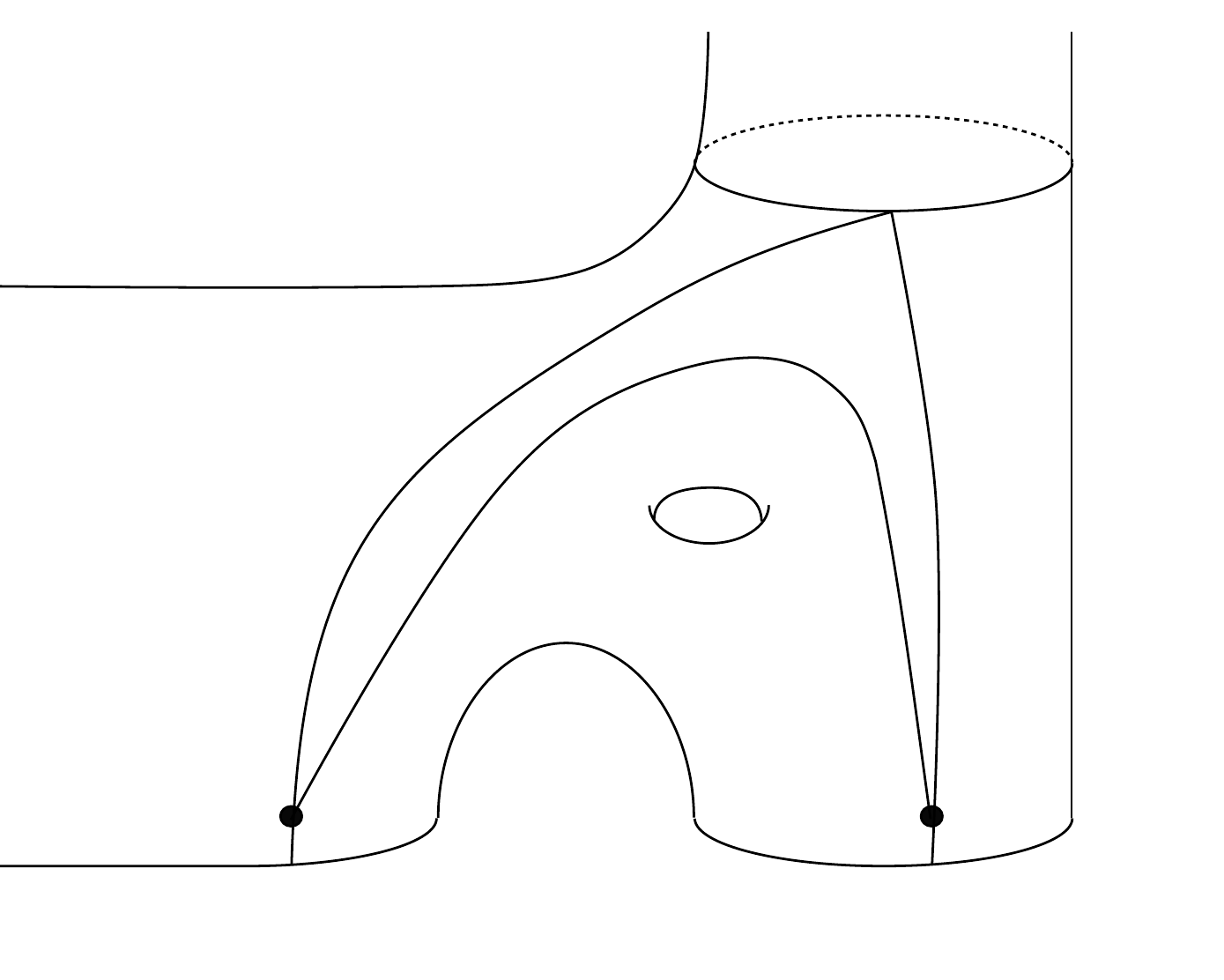
\caption{}
\label{fig:arcinvt}
\end{figure}

By Theorem \ref{thm:markov}, positive Markov stabilizations have no effect on the transverse closure of a pointed monodromy. 
Let $\phi$ denote the open book monodromy. By construction, $T(L,\alpha)$ is the transverse closure of $\phi\circ \pi_\alpha$ and $T(L,\alpha')$ is the transverse closure of $\phi \circ \pi_{\alpha'}$.
Markov stabilizing $\phi\circ \pi_\alpha$ using the arc $\gamma$ we obtain $\phi\circ \pi_\alpha\circ \sigma_\gamma$.

Conjugation has no effect on the braid isotopy class, hence no effect on the transverse isotopy class.
By conjugating, Markov stabilizing, and conjugating again we may pass from $\phi \circ \pi_{\alpha'}$ to $\phi \circ \sigma_\gamma \circ \pi_{\alpha'}$.

Let $S$ denote a neighborhood of $L\cup \alpha\cup \alpha'$, equipped with the two marked pointed $p$ and $p'$. The monodromies $\pi_\alpha \circ \sigma_{\gamma}$ and $\sigma_\gamma \circ \pi_{\alpha '}$ are supported in $S$, we claim they are isotopic. Consider the three properly embedded arcs $\{\beta_1,\beta_2,\beta_3\}$ in $S\smallsetminus \{p\cup p'\}$  pictured on the right side of Figure \ref{fig:arcinv2}. 
The pointed monodromies $\pi_\alpha \circ \sigma_{\gamma}$ and $\sigma_\gamma \circ \pi_{\alpha '}$ 
act identically, up to isotopy, on the $\beta$-arcs; their images are pictured on the right of Figure \ref{fig:arcinv2} as well. Since the complement of these arcs is a disk, the claim follows.

\begin{figure}[h]
\def\svgwidth{250pt}
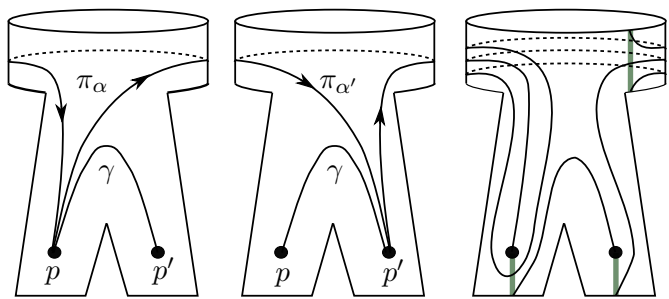
\caption{On the far right, the set of arcs $\{\beta_i\}$ is lightly shaded.}
\label{fig:arcinv2}
\end{figure}

In summary, $T(L,\alpha)$ is transversely isotopic to the closure of $\phi\circ \pi_\alpha\circ \sigma_\gamma$, which is braid isotopic to the closure of $\phi \circ \sigma_\gamma \circ \pi_{\alpha'}$, which itself is transversely isotopic to $T(L,\alpha')$.


\end{proof}

\begin{lemma}
\label{sequence}
Let $L\subset \Sigma = \pi^{-1}(0)$ be a boundary compatible knot. Suppose $\alpha,\alpha'\subset \Sigma\smallsetminus L$ are auxiliary arcs. 
There exists a sequence of auxiliary arcs 
\[
\alpha = a_0,a_1,\dots,a_m = \alpha '
\]
such that any two consecutive arcs $a_{j}$ and $a_{j+1}$ can be made disjoint via isotopy through auxiliary arcs.
\end{lemma}
\begin{proof}
We construct a surface with with connected boundary $-L$ in the following way.
If $L$ is separating, consider the connected component of $\Sigma \smallsetminus L$ containing $\partial \Sigma$. By capping off each component of $\partial \Sigma$ with a disk containing a marked point, we obtain a surface $S$ with connected boundary $-L$.
If $L$ is non-separating, we cap off all boundary components of $\Sigma \smallsetminus L$ in the same manner except the one identified with $-L$, and let $S$ denote the resulting surface.
Let $\{p_1,\dots,p_n\}$ denote the resulting marked points.

Under this procedure, an auxiliary arc $c$ becomes a properly embedded arc $r(c)$ 
connecting $-L = \partial S$ to a point of $\{p_1,\dots,p_n\}$.
Given such an arc $r(c)$, we may reverse the procedure and recover $c$ up to isotopy through auxiliary arcs.
Two auxiliary arcs $a$ and $a'$ can be made disjoint via isotopy through auxiliary arcs if the arcs $r(a)$ and $r(a')$ can be made disjoint along their interiors via relative isotopy.

Isotope the auxiliary arcs $\alpha$ and $\alpha'$ to meet at a point $\{e\}\in L$.
By Lemma 4.3 of \cite{braiddynamics} there exists a sequence of properly embedded arcs 
\[
r(\alpha) = r(a_0),r(a_1),\dots,r(a_m) = r(\alpha')
\]
in $S\smallsetminus \{p_1,\dots,p_n\}$ connecting $\partial S$ to a point of $\{p_1,\dots,p_n\}$ so that any two consecutive arcs $r(a_j)$ and $r(a_{j+1})$ are disjoint along their interiors. 
Reversing the procedure above, this sequence gives rise to the sequence of auxiliary arcs $\alpha = a_0,a_1,\dots,a_m = \alpha '$ desired.
\end{proof}

\begin{proof}[Proof of Proposition \ref{uniquetrans}]

Let $\phi$ denote the monodromy of an open book $(B,\pi)$ having page $\pi^{-1}(0) = \Sigma$.
Let $L= L_1\cup\dots\cup L_n\subset \Sigma$ be a boundary compatible link equipped with two sets of auxiliary arcs $\{\alpha_i\}$ and $\{\alpha_i '\}$. 
$T(L,\{\alpha_i\})$ is defined to be the transverse closure of $\phi \circ \prod_{i=1}^{n} \pi_{\alpha_i}$, and $T(L,\{\alpha_i '\})$ the transverse closure of $\phi \circ \prod_{i=1}^{n} \pi_{\alpha_i '}$. We claim these are transversely isotopic. Let 
\[
\alpha_n = a_0,\dots,a_m=\alpha_n '
\]
denote the sequence of auxiliary arcs from $L_n$ to $\partial \Sigma$ specified by Lemma \ref{sequence}.

As in the proof of Lemma \ref{disjoint}, 
by applying sequence of conjugations and positive Markov (de)stabilizations we pass from \[\phi \circ \prod_{i=1}^{n} \pi_{\alpha_i} =\phi\circ \pi_{\alpha_1}\circ\dots\circ \pi_{\alpha_{n-1}}\circ \pi_{a_0}\hspace{1cm}  \text{to} \hspace{1cm} \phi\circ \pi_{\alpha_1}\circ\dots\circ \pi_{\alpha_{n-1}} \circ \pi_{a_1}.\] Repeating this procedure a total of $m$ times, we pass to 

\[\phi\circ \pi_{\alpha_1}\circ\dots\circ \pi_{\alpha_{n-1}}\circ \pi_{a_m} = \phi\circ \pi_{\alpha_1}\circ\dots\circ \pi_{\alpha_{n-1}} \circ \pi_{\alpha_n '}.\]

Via conjugations and repetition of the procedure above, we may replace each $\pi_{\alpha_i}$ with $\pi_{\alpha_i '}$ in the monodromy. Since none of these moves have any effect on the transverse isotopy class of the closure, the claim is established.
\end{proof}


We turn to the other part of Theorem \ref{uniquetransapprox}:

\begin{proposition}
\label{approx}
Let $L$ be a non-isolating oriented link in the page $\Sigma = \pi^{-1}(0)$ of an open book $(B,\pi)$. Then
the positive transverse push-off of the Legendrian representative of Theorem \ref{uniqueleg} is transversely isotopic to transverse representative of Proposition \ref{uniquetrans}, i.e. $L_+ = T(L)$.

\end{proposition}
\begin{proof}
We prove the result in the case that $L$ is a knot. 
Let $\alpha$ denote an auxiliary arc for $L\subset \Sigma$, $N$ a neighborhood of $L\cup \alpha$ in $\Sigma$, and $\delta'$ the boundary component of $N$ meeting $\partial \Sigma$. Let $\delta$ be the closure of the intersection of $\delta '$ with the interior of $\Sigma$. 

\begin{figure}[h]
\def\svgwidth{250pt}
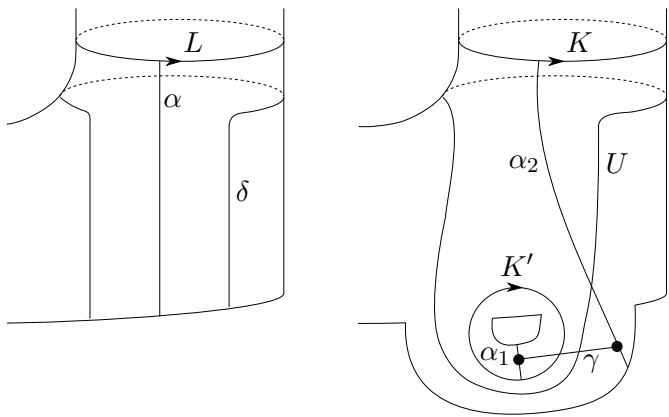
\caption{The page, before and after stabilization. $\Sigma\to \Sigma '$.}
\label{fig:unique}
\end{figure}

The properly embedded arc $\delta$ guides a positive open book stabilization. 
Let $\Sigma '$ denote the new page, obtained by attaching a 1-handle to $\Sigma$ with feet at the endpoints of $\delta$; let $U\subset \Sigma'$ denote the union of $\delta$ with the core of this 1-handle.
The new monodromy $\tau_U\circ \phi$ is obtained by composing the original monodromy $\phi$ with a Dehn twist about $U$.


To avoid confusion, we let $K$ denote a copy of $L$ sitting in $\Sigma'$; $L=K$ as Legendrian knots.
Since the arc $\delta$ is disjoint from the auxiliary arc $\alpha$, it is clear that $T(L)=T(K)$; see for example Corollary 2.5 of \cite{equiv}.

Sliding $K$ over $U$ one obtains $K'$, see the right-hand side of Figure \ref{fig:unique}. It has been observed in Lemma 6.5 of \cite{CCSMCM} that $K'$ is a Legendrian stabilization $S_- (K)$, and in particular that the binding component $T\subset \partial \Sigma '$ parallel to $K'$ is the transverse push-off of $K$. In Section 2.6 of \cite{braiddynamics}, we explain how to transversely isotope a binding component of an open book to obtain an index one braid, our construction shows that $T(K') = T$.


In summary, we wish to establish that $T(L) = L_+$, and we have that
\begin{gather*}
T(L) = T(K) \hspace{1cm} \text{and}\hspace{1cm}T(K') = T = K_+ = L_+.
\end{gather*}
It remains to show that $T(K) = T(K')$. 
Consider the auxiliary arcs $\alpha_1$ and $\alpha_2$ for $K'$ and $K$, respectively, pictured in the right-hand side of Figure \ref{fig:unique}. 
Let $p_i\in \alpha_i$ denote the marked point introduced in the construction of the push-map $\pi_{\alpha_i}$.
$T(K)$ is the closure of $\phi \circ \pi_{\alpha_2}\circ \tau_U \in Mod(\Sigma'\smallsetminus \{p_2\},\partial\Sigma')$, whereas $T(K')$ is the closure of $\phi \circ \pi_{\alpha_1}\circ \tau_U\in Mod(\Sigma'\smallsetminus \{p_1\},\partial\Sigma')$. 

Let $f$ denote a diffeomorphism of $\Sigma '$ which is
isotopic to the identity, sends $p_2$ to $p_1$, and is supported in a neighborhood of the arc $\gamma$, also pictured in Figure \ref{fig:unique}.
We claim that $\phi \circ \pi_{\alpha_1}\circ \tau_U$ and $f\circ \phi \circ \pi_{\alpha_2}\circ \tau_U\circ f^{-1}$ are equal as elements of $Mod(\Sigma'\smallsetminus \{p_1\},\partial\Sigma')$, and hence have braid isotopic closures, implying that $T(K) = T(K')$.
Note that $\phi$ is supported away from $\gamma$, and thus commutes with $f$. It suffices to show that $\pi_{\alpha_1}\circ \tau_U$ and $f\circ \pi_{\alpha_2}\circ \tau_U\circ f^{-1}$ are isotopic.

Consider a pair of pants $P\subset \Sigma '$ obtained by slightly enlarging the region $N\subset \Sigma$ and attaching the 1-handle $\Sigma ' \smallsetminus \Sigma$. The three boundary components of $P$ are parallel to $K, K'$ and $U$. The intersection of $P$ with a collar neighborhood of $\partial \Sigma '$ has two components, an annulus and a disk; see Figure \ref{fig:unique2}. The pointed monodromies $\pi_{\alpha_1}\circ \tau_U$ and $f\circ \pi_{\alpha_2}\circ \tau_U\circ f^{-1}$ are supported in $P$. 

\begin{figure}[h]
\def\svgwidth{400pt}
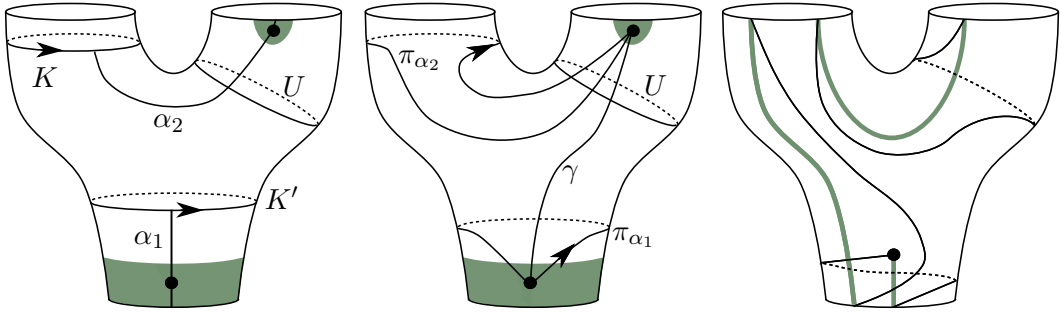
\caption{The pair of pants $P$. Left and center, the intersection of $P$ with a collar neighborhood of $\partial \Sigma '$ is shaded. On the right, the arcs $\{\beta_i\}$ are lightly shaded, their images appear in black.}
\label{fig:unique2}
\end{figure}

The set of properly embedded arcs $\{\beta_1,\beta_2,\beta_3\}\subset P\smallsetminus \{p_1\}$ pictured on the right-hand side of Figure \ref{fig:unique2}, cut the surface into a disk having no marked points.
Both of the pointed monodromies act on the arcs $\{\beta_i\}$ identically, their images are also pictured, hence the monodromies are isotopic and the result follows.

It is clear that one can generalize this argument to the case of a link, since the entire construction occurs in a neighborhood of a component and its auxiliary arc.

\end{proof}

\section{Knot Floer homology and invariants of Legendrian and transverse links}

\subsection{Knot Floer homology}
\label{subsec:hfk}
We fix some notation for knot Floer chain complexes and homology groups \cite{holknots}. We work with $\mathbb{F} = \mathbb{Z}_2$ coefficients throughout this paper.

To a $2n$-pointed Heegaard diagram $\mathcal{H} = (\Sigma,\boldsymbol{\alpha},\boldsymbol{\beta},\bold{w},\bold{z})$, encoding an $m$-component link $L\subset Y$, we associate the \emph{minus} chain complex
$CFK^- (\mathcal{H})$. As an $\mathbb{F} [U_1,\dots,U_n]$-module, this complex is generated by tuples of the form
$[\bold{a},0,\dots,0]$, where
$\bold{a}\in \mathbb{T}_{\boldsymbol{\alpha}}\cap\mathbb{T}_{\boldsymbol{\beta}}$. 
 The action of the $U$-variables is specified by \[U_j \cdot [\bold{a},i_1,\dots,i_j,\dots,i_n] = [\bold{a},i_1,\dots,i_j -1,\dots, i_n].\]  

The differential counts pseudo-holomorphic disks 
\[
\partial ^- [\bold{a},i_1,\dots,i_n] = 
\sum\limits_{ \bold{b}\in \mathbb{T}_{\boldsymbol{\alpha}}\cap\mathbb{T}_{\boldsymbol{\beta} }} \sum\limits_{\substack{\phi\in\pi_2 (\bold{a},\bold{b})\\ \mu(\phi)=1\\  \{n_w(\phi) = 0| w\in\bold{w} \}}}  (\# \widehat{\mathcal{M}}(\phi))  [\bold{b},i_1-n_{z_1}(\phi),\dots,i_n-n_{z_n}(\phi)].
\]
If two basepoints $z_i$ and $z_j$ correspond to the same component of $L$, then the variables $U_i$ and $U_j$ act identically on $HFK^-(Y,L):= H_* (CFK^- (\mathcal{H}))$.
By reindexing the basepoints, we arrange that $z_1,\dots, z_m$ correspond to the $m$ distinct components of the link. $HFK^-(Y,L)$ is then an invariant of the pair $(Y,L)$, well-defined up to $\mathbb{F} [U_1,\dots, U_m]$-module isomorphism. 

Quotienting $CFK^-(\mathcal{H})$ by the images of $\{U_i\}_{i=1}^m$
, we obtain the \emph{hat} complex $\widehat{CFK}(\mathcal{H})$. The homology of this complex, denoted $\widehat{HFK}(Y,L)$, is an invariant of the pair $(Y,L)$ up to $\mathbb{F}$-module isomorphism.



Alternatively, if we quotient $CFK^-(\mathcal{H})$ by $\{U_i\}_{i=1}^n$ to set all of the $U$-variables equal to zero, we obtain the \emph{totally blocked} complex $\widetilde{CFK}(\mathcal{H})$, whose homology $\widetilde{HFK}(\mathcal{H})$ is an invariant of $(Y,L,n)$, where $n$ is the number of basepoint pairs. As $\mathbb{F}$-modules, \[\widetilde{HFK}(\mathcal{H})\simeq \widehat{HFK}(Y,L)\otimes W^{\otimes n-m},\] where $W$ is a 2-dimensional $\mathbb{F}$ vector space. If $Y$ is a $\mathbb{Q}HS^3$ then the two generators of $W$ have Maslov-Alexander bigradings $(0,0)$ and $(-1,-1)$. 
The projection $\widehat{CFK}(\mathcal{H})\to \widetilde{CFK}(\mathcal{H})$ induces an injection $\iota :\widehat{HFK}(Y,L)\xhookrightarrow{} \widetilde{HFK}(\mathcal{H})$.


Finally, quotienting $HFK^-(Y,L)$ by the image of $(U_i - U_j)$, for every $1\le i,j\le m$ has the effect of identifying all of the $U$-variable actions. We let $HFK^{=}(Y,L)$ denote the resulting $\mathbb{F} [U]$-module. 

\subsection{Invariants of Legendrian and transverse links}
\label{subsec:invts}
Baldwin, Vela-Vick and V\'{e}rtesi \cite{equiv} used Theorem \ref{thm:markov} to define an invariant of transverse links in a contact 3-manifold, taking values in the minus and hat versions of knot Floer homology.
To a transverse link $K\subset (Y,\xi)$ they associate classes $t^{\circ}(K)\subset HFK^{\circ}(-Y,K)$, where $\circ \in \{\wedge,-\}$. 

Their construction is reminiscent of Honda, Kazez and Matic's \cite{HKM} reformulation of the Ozsv\'{a}th-Szab\'{o} \cite{contactclass} contact invariant $c(\xi)$.
Picking a pointed monodromy having transverse closure $K$, \cite{equiv} construct a multi-pointed Heegaard diagram for $(-Y,K)$ whose associated knot Floer complex is equipped with a natural generator representing $t^\circ$.
We review their construction.

Let $(B,\pi)$ be an open book supporting $(Y,\xi)$. Let $\Sigma_\theta = \pi^{-1}(\theta)$ be a page. 
Let $K$ be an index $n$ braid about $(B,\pi)$, $K\cap \Sigma_0=\{p_1,\dots,p_n\}$, and let $\phi \in Mod(\Sigma\smallsetminus \{p_1,\dots,p_n\},\partial \Sigma)$ denote a pointed monodromy specifying $K$.

A \emph{basis of arcs} $\{a_i\}_1 ^{m}\subset \Sigma\smallsetminus \{p_1,\dots,p_n\}$ is a collection of properly embedded disjoint arcs which cut $\Sigma\smallsetminus \{p_1,\dots,p_n\}$ into $n$ discs, each containing one marked point. 
Let $\{b_i\}_1^{m}$ be another basis of arcs obtained by perturbing the endpoints of $a_i$ in the oriented direction of $\partial \Sigma$, and isotoping in $\Sigma\smallsetminus \{p_1,\dots,p_n\}$ so that $a_i$ intersects $b_i$ transversely in a single point with positive sign. 
The pointed monodromy $\phi$ together with any basis of arcs specifies a Heegaard diagram $\mathcal{H}_\phi = (S,\boldsymbol{\beta},\boldsymbol{\alpha},\bold{w},\bold{z})$ encoding $(-Y,K)$:
\begin{itemize}
\item
$S=\Sigma_{1/2}\cup -\Sigma_0$
\item
$\alpha_i = a_i\times \{0,1/2\}$, $\beta_i = b_i\times \{1/2\}\cup \phi(b_i)\times \{0\}$
\item
$z_i = p_i \times \{0\}$, $w_i = p_i \times \{1/2\}$
\end{itemize}

For each $i$, $\alpha_i$ intersects $\beta_i$ in a single point in the region $\Sigma_{1/2}$ denoted $x_i$. Let $\bold{x_\phi} \in \mathbb{T}_{\boldsymbol{\beta}} \cap \mathbb{T}_{\boldsymbol{\alpha}}$ denote the generator having component $x_i$ on $\alpha _i$. The class $t^\circ(K)$ is the homology class $[\bold{x}_\phi] \in HFK^\circ (\mathcal{H}_\phi) = HFK^\circ (-Y,K)$.

The proof of invariance utilizes the Giroux correspondence \cite{giroux} and Theorem \ref{thm:markov}. Any two braidings of a transverse link $K$ about two open books are related by a sequence of braid isotopies, positive Markov stabilizations of the braids, and positive Hopf stabilizations of the open books. Each of these moves, along with arc slides relating bases of arcs, give rise to a quasi-isomorphisms of the underlying $CFK^\circ$ complexes preserving the distinguished generator $\bold{x}_\phi$, see \cite{equiv} for details. 

Given the pointed monodromy $\phi$, we can also consider the homology class of $\bold{x}_\phi \in \widetilde{CFK}(\mathcal{H}_\phi)$, denoted $\widetilde{t}(\phi)\in \widetilde{HFK}(\mathcal{H}_\phi)$. The proof of arc-slide invariance for the invariant $t^\circ$ also shows that $\widetilde{t}$ is an invariant of the pointed monodromy $g$. 

\begin{lemma}
\label{tildenonvanishing}
Suppose $g \in Mod(S\smallsetminus P,\partial S)$ has transverse closure $K_g\subset (Y_g,\xi_g)$. Then
\[
\widehat{t}(K_g)\ne 0 \iff \widetilde{t}(g)\ne 0.
\]
\end{lemma}
\begin{proof}
The natural injection $\iota: \widehat{HFK}(-Y_g,K_g)\xhookrightarrow{} \widetilde{HFK}(\mathcal{H}_g)$ maps $\widehat{t}(K_g)$ to $\widetilde{t}(g)$.
\end{proof}

Two Legendrian links are Legendrian isotopic after some number of negative Legendrian stabilizations if and only if their transverse push-offs are transversely isotopic \cite{transverseapprox}; invariants of transverse links are equivalent to invariants of Legendrian links which are preserved under negative stabilization. For a Legendrian link $L$, we set $t^{\circ}(L):= t^\circ(L_+)$.

There are two other notable invariants of Legendrians taking values in knot Floer homology. Ozsv\'ath, Sz\'abo and Thurston \cite{grid} used grid diagrams to define a combinatorial invariant $\lambda^+$ of Legendrian links in the tight 3-sphere. Lisca, Ozsv\'ath, Sz\'abo and Stipsicz \cite{LOSS} used the Giroux correspondence to define an invariant $\mathcal{L}$ of Legendrian \emph{knots} in arbitrary contact 3-manifolds.
These invariants are equivalent whenever they are defined:

\begin{theorem}\cite{equiv}
\label{equiv}
\begin{enumerate}
\item
If $K\subset (S^3,\xi_{std})$ is an $m$-component Legendrian link, there exists an $\mathbb{F}[U_1,\dots,U_m]$-module automorphism of $HFK^-(-S^3,K)$
sending $\lambda^+ (K)$ to $t^- (K)$. 
\item
If $K\subset (Y,\xi)$ is a Legendrian knot, there exists an $\mathbb{F}[U]$-module automorphism of
$HFK^-(-Y,K)$ sending $\mathcal{L} (K)$ to $t^- (K)$.
\end{enumerate}

\end{theorem}

\subsection{Comultiplicativity of the transverse invariant}
\label{seccomult}

In \cite{braiddynamics}, the author established comultiplicativity of the transverse invariant, generalizing \cite{comultcontact} and \cite{comultgrid}.
Let \[K_g\subset (Y_g,\xi_g), K_h\subset (Y_h,\xi_h), \hspace{.1cm} \text{and}\hspace{.1cm} K_{hg}\subset (Y_{hg},\xi_{hg})\] be transverse links specified as closures of pointed monodromies \[g,h, hg\in Mod(S\smallsetminus \{p_1,\dots,p_n\},\partial S).\] 

\begin{theorem}
\label{thm:comultiplication}
There exists a natural comultiplication homomorphism
\[
\mu ^-:HFK^-(-Y_{hg},K_{hg})\to HFK^-(-Y_{g},K_{g})\otimes HFK^-(-Y_{h},K_{h})
\]
of $\mathbb{F}[U_1,\dots,U_n]$-modules
sending $t^-(K_{hg})$ to $t^-(K_g) \otimes t^- (K_h)$. 
\end{theorem}

Each of these groups have the structure of an $\mathbb{F}[U_1,\dots,U_n]$-module, coming from the $n$ pairs of basepoints in the diagram used to define each of the invariants.
The actions of $U_i$ and $U_j$ agree, on $HFK^-(-Y,K)$, if the marked points $p_i$ and $p_j$ correspond to the same component of $K$, see Remark \ref{hatcomult}. 

Recall that quotienting $HFK^-$ by the images of $\{U_i - U_j | 1\le i,j\le n\}$ we obtain $HFK^=$, for a knot these theories are equivalent.
Let $t^= (K)\in HFK^= (-Y,K)$ denote the image of $t^-(K)$ under the quotient $HFK^- (-Y,K)\to HFK^= (-Y,K)$. The following is immediate:
\begin{corollary}
\label{comult=}
There exists a natural comultiplication homomorphism
\[
\mu ^=:HFK^=(-Y_{hg},K_{hg})\to HFK^=(-Y_{g},K_{g})\otimes HFK^=(-Y_{h},K_{h})
\]
of $\mathbb{F}[U]$-modules
sending $t^=(K_{hg})$ to $t^=(K_g) \otimes t^= (K_h)$.
\end{corollary}

There is a version of comultiplicativity for the totally blocked theory:
\begin{theorem}
\label{thm:comult}
There exists a natural comultiplication homomorphism
\[
\widetilde{\mu}_*: \widetilde{HFK}(\mathcal{H}_{hg})\to \widetilde{HFK}(\mathcal{H}_g)\otimes \widetilde{HFK}(\mathcal{H}_h),
\]
sending $\widetilde{t}(hg)$ to $\widetilde{t}(g)\otimes \widetilde{t}(h)$.
In particular, $\widehat{t}(K_g),\widehat{t} (K_h) \ne 0$ implies that $\widehat{t}(K_{hg})\ne 0$.
\end{theorem}

\begin{remark}
\label{hatcomult}
If the three pointed monodromies $g,h$ and $hg$ each have closures which are $n$-component links, i.e. if they are \emph{pure braids}, then the map of Theorem \ref{thm:comultiplication} is particularly nice, because the three $HFK^-$ groups are naturally $\mathbb{F}[U_1,\dots,U_n]$ modules, i.e. the $U$-variable actions are all distinct. 
Also, in this case the tilde and hat complexes coincide, and the map of Theorem \ref{thm:comult} is a map between $\widehat{HFK}$ groups sending $\widehat{t}(K_{hg})\to \widehat{t}(K_{g})\otimes \widehat{t}(K_{h})$.

\end{remark}

The closure of a right-hand Dehn or half twist has non-vanishing hat invariant:

\begin{proposition} \cite{braiddynamics}
\label{prop:nonzero}
Let $\Sigma$ be an oriented, compact surface with boundary and some marked points. Let $K_\gamma$ denote the closure of a right-hand half twist along some embedded arc $\gamma$ connecting two marked points. Let $K_\delta$ denote the closure of a right-hand Dehn twist along any simple closed curve $\delta$. Then $\widehat{t}(K_\gamma),\widehat{t}(K_\delta) \ne 0.$
\end{proposition}


\subsection{Natural quotients}
\label{nat}

We fix notation for quotient maps associated to removing $sl=-1$ unknots and tight $S^1\times S^2$-summands which, combined with comuliplicativity, will be used to prove all of the naturality statements of the transverse invariant in this paper. 
One should think of these maps as contravariantly induced by births of $sl=-1$ unknots or Weinstein 1-handle attachments.

Suppose $(B,\pi)$ be an open book with page $\Sigma=\pi^{-1}(0)$ monodromy $\phi$. Introduce $n$ marked points $\{p_1,\dots,p_n\}$ in a collar neighborhood of $\partial \Sigma$ which is fixed by $\phi$. 
Let \[\phi _n\in Mod(\Sigma\smallsetminus \{p_1,\dots, p_n\}, \partial \Sigma)\] denote the trivial extension of $\phi$, and $U_n$ the transverse closure of $\phi _n$. 
$U_n$ consists of $n$ disjoint positively braided meridional unknots of the binding $B$.

\begin{lemma}
\label{meridians}
There are $\mathbb{F}[U]$ and $\mathbb{F}$-module isomorphisms
\begin{gather*}
g^-: HFK^-(-Y,U_1)\to \widehat{HF}(-Y)\otimes \mathbb{F}[U]\\
\widehat{g}:\widehat{HFK}(-Y,U_1)\to \widehat{HF}(-Y)
\end{gather*}
mapping $t^-(U_1)\to c(\xi)\otimes 1$ and $\widehat{t}(U_1)\to c(\xi)$, respectively.
\end{lemma}
\begin{proof}

Note that a basis of arcs for $\Sigma$ is a basis of arcs for $\Sigma \smallsetminus \{p\}$, in particular the diagrams used to define the transverse and contact invariants differ by a single basepoint.

Let $\mathcal{H}:= \mathcal{H}_{\phi_1} = (S,\boldsymbol{\beta},\boldsymbol{\alpha},w,z)$ be the doubly pointed Heegaard diagram for $(-Y,U_1)$ constructed in Section \ref{subsec:invts}, and $\mathcal{H}' := (S,\boldsymbol{\beta},\boldsymbol{\alpha},w)$.
Note that the $w$ and $z$ basepoints are in the same component of $S\smallsetminus \{ \boldsymbol{\alpha} \cup \boldsymbol{\beta}\}$, thus as complexes $CFK^- (\mathcal{H}) = \widehat{CF}(\mathcal{H}')\otimes \mathbb{F}[U]$ and $\widehat{CFK}(\mathcal{H}) = \widehat{CF}(\mathcal{H}')$.
The induced isomorphisms on homology map $t^-(U_1)\to c(\xi)\otimes 1$ and $\widehat{t}(U_1)$ to $c(\xi)$ (see the \cite{HKM} definition of the contact invariant).

\end{proof}

Suppose now that $K$ is braided about $(B,\pi)$ with index $n$, and is specified by some pointed monodromy $\psi$. Introduce a single new marked point $p_{n+1}$ in a collar neighborhood
of the boundary and extend $\psi$ trivially to a pointed monodromy
\[
\psi ' \in Mod(\Sigma\smallsetminus \{p_1,\dots, p_{n+1}\}, \partial \Sigma).
\]
having transverse closure $K\cup U_1$.

\begin{lemma}
\label{birth}
There are $\mathbb{F}$ and $\mathbb{F}[U]$-module homomorphisms
\[
h^\circ: HFK^\circ (-Y,K\cup U_1)\to HFK^\circ (-Y,K)
\]
mapping $t^{\circ}(K\cup U_1)\to t^{\circ}(K)$, where $\circ\in\{\wedge,=\}$.
\end{lemma}
\begin{proof}
We may complete any basis 
for $\Sigma\smallsetminus \{p_1,\dots, p_n\}$ to a basis 
for $\Sigma\smallsetminus \{p_1,\dots,p_n,p_{n+1}\}$ by adding boundary a single boundary parallel arc 
which is fixed by $\psi'$ and cobounds a disk with a sub-arc of $\partial \Sigma$ containing only the marked point $p_{n+1}$.

As outlined above, to this data \cite{equiv} associate a multi pointed Heegaard diagram $\mathcal{H}_{\psi '} = (S,\boldsymbol{\beta},\boldsymbol{\alpha},\bold{w}\cup w_u,\bold{z}\cup z_u)$. The pair of basepoints $\{w_u,z_u\}$ encoding $U_1$ are in the same component of $S\smallsetminus \{ \boldsymbol{\alpha} \cup \boldsymbol{\beta}\}$. Removing the basepoint $z_u$ does not change the differential of the $CFK^-$ complex. The resulting diagram has a single free basepoint which can be removed via destabilization. The destabilization chain map induces the desired map on homology. 

\end{proof}


In the case that the monodromy of an open book is trivial, the closure is a connected sum of $m= -\chi (\Sigma)$ copies of $S^1\times S^2$ and the supported contact structure $\xi_{std}$ has non-vanishing contact invariant, $c(\xi_{std})\ne 0$.
\begin{lemma}
\label{lemqpm}
Let $U_n\subset (\#^m (S^1\times S^2),\xi_{std})$ be the transverse link specified above. Then
\begin{gather*}
\widehat{t}(U_n), t^=(U_n),t^-(U_n)\ne 0.
\end{gather*}
Moreover, $t^-(U_n)$ is not in the image of $U_i$, for $1\le i \le n$.
\end{lemma}
\begin{proof}
Let $\mathcal{H} = (S,\boldsymbol{\beta},\boldsymbol{\alpha},\bold{w},\bold{z})$ be the Heegaard diagram used in the definition of the transverse invariant coming from any basis of arcs for $\Sigma\smallsetminus \{p_1,\dots,p_n\}$. Since the pointed monodromy is trivial, the diagram gives rise to a $CFK^-$ complex with trivial differential. 
\end{proof}

This allows us to define projections:
\begin{gather*}
\widehat{G} :\widehat{HFK}(-\#^m(S^1\times S^2), U_n)\to \langle \widehat{t}(U_n)\rangle \simeq \mathbb{F},\\
G^= :HFK^=(-\#^m(S^1\times S^2), U_n)\to \langle t^=(U_n)\rangle \simeq\mathbb{F}[U],\\
\text{and}\hspace{.5cm} G^- :HFK^-(-\#^m(S^1\times S^2), U_n)\to \langle t^-(U_n)\rangle \simeq\mathbb{F}[U_1,\dots,U_n].
\end{gather*}

\subsection{Contact connected sums}
\label{conect}
There is a well defined notion of connected sum, in fact a prime decomposition theorem holds, in the contact category \cite{colin}. Given contact manifolds $(Y_1,\xi_1)$ and $(Y_2,\xi_2)$, we may take their connected sum to obtain a contact manifold $(Y_1\# Y_2,\xi_1\# \xi_2)$. Let $(\Sigma_i,\phi_i)$ be abstract open books giving rise to $(Y_i,\xi_i)$, and 
$\Sigma_1\sharp \Sigma_2$ denote the boundary connected sum of surfaces. $(\Sigma_1 \sharp \Sigma_2, \phi)$ is an abstract open book for the contact connected sum, where $\phi$ is the natural monodromy restricting to $\phi_i$ on $\Sigma_i$. This has an obvious extension to pointed monodromies if one wishes to keep track of transverse links within the summands. 4-dimensionally, this procedure corresponds to a Weinstein 1-handle attachment.

\begin{lemma}
\label{connectedsum}
Suppose that $K_1\subset (Y_1,\xi_1)$ and $K_2\subset(Y_2,\xi_2)$ are transverse links, then there exist homomorphisms of $\mathbb{F}, \mathbb{F}[U]$ and $\mathbb{F}[U_1,\dots,U_n]$-modules
\begin{gather*}
\widehat{HF} (-(Y_1\# Y_2))\to \widehat{HF} (-Y_1)\otimes \widehat{HF} (-Y_2)\\
HFK^\circ (-(Y_1\# Y_2),K_1)\to HFK^\circ (-Y_1,K_1)\otimes \widehat{HF} (-Y_2)\\
HFK^\circ (-(Y_1\# Y_2),K_1\cup K_2)\to HFK^\circ (-Y_1,K_1)\otimes HFK^\circ (-Y_2,K_2)
\end{gather*}
mapping 
\begin{gather*}
c(\xi_1 \# \xi_2)\to c(\xi_1)\otimes c(\xi_2)\\
t^\circ (K_1)\to t^\circ (K_1)\otimes c(\xi_2)\\
\text{and} \hspace{.5cm} t^\circ(K_1\cup K_2)\to t^\circ (K_1)\otimes t^\circ (K_2),
\end{gather*}
respectively. Here $\circ\in \{\wedge,=,-\}$.
\end{lemma}
\begin{proof}
The constructions of the first two maps are identical, we construct the second. Let $g \in Mod(\Sigma \smallsetminus P, \partial \Sigma)$ denote a pointed monodromy having transverse closure $K_1\subset (Y_1,\xi_1)$, and $(S,\phi)$ an abstract open book specifying $(Y_2,\xi_2)$. There is a natural pointed monodromy $h\in Mod (\Sigma \sharp S \smallsetminus P, \partial (\Sigma \sharp S))$ which restricts to $g$ and $\phi$ on the two respective summands. By construction, $h$ has transverse closure $K_1 \subset (Y_1\# Y_2,\xi_1 \# \xi_2)$.

The union of any bases for $\Sigma\smallsetminus P$ and $S$ forms a basis for $\Sigma \sharp S \smallsetminus P$. In particular, Heegaard diagrams $\mathcal{H}_g$ and $\mathcal{H}_\phi$ used in the definition of $t^\circ (K_1)\subset HFK^\circ (-Y_1,K_1)$ and $c(\xi_2)$ can be connect summed to obtain a Heegaard diagram $\mathcal{H}_h$ defining\footnote{after identifying two of the $\bold{w}$-basepoints.} $t^\circ (K_1)\subset HFK^\circ (-(Y_1\#Y_2),K_1)$. There is an identification of chain complexes $CFK^-(\mathcal{H}_h) = CFK^- (\mathcal{H}_g)\otimes \widehat{CF}(\mathcal{H}_\phi)$ relating the generators $\bold{x}_h$ and $\bold{x}_g\otimes \bold{x}_\phi$, whose homology classes are the relevant invariants.

The construction of the third map is similar. Let $g'\in Mod(S\smallsetminus P',\partial S)$ denote a pointed monodromy having transverse closure $K_2\subset (Y_2,\xi_2)$. Let $h' \in Mod(\Sigma\sharp S \smallsetminus (P\cup P'),\partial (\Sigma\sharp S))$ denote the natural pointed monodromy which restricts to $g$ and $g'$ on the two respective summands. By construction, $h'$ has transverse closure $K_1\cup K_2\subset (Y_1\# Y_2,\xi_1\# \xi_2)$. The union of bases for $\Sigma\smallsetminus P$ and $S\smallsetminus P'$ can be extended to a basis for $\Sigma\sharp S \smallsetminus (P\cup P')$ by adding a single arc which is fixed by $h'$. By considering such a basis, we obtain a Heegaard diagram $\mathcal{H}_{h'}$ defining the invariant $t^\circ (K_1\cup K_2)$, this coincides with taking the connected sum of diagrams $\mathcal{H}_g$ and $\mathcal{H}_{g'}$ and adding a single $\alpha$/$\beta$-curve pair, coming from the new basis arc fixed by $h'$, giving rise to an identification of chain complexes
\[
CFK^- (\mathcal{H}_{h'})= CFK^- (\mathcal{H}_g)\otimes CFK^- (\mathcal{H}_{g'})\otimes W
\]
where $W= \langle a,b\rangle$ is a 2-dimensional $\mathbb{F}$-vector space with trivial differential. Under this identification $\bold{x}_{h'} \to \bold{x}_g \otimes \bold{x}_{g'}\otimes a$. Taking homology and projecting $W\to \langle a\rangle$ gives the desired map relating the invariants.
\end{proof}

\section{Quasipositive surfaces and comultiplication maps}
\label{secqpm}



Theorem \ref{qpm} follows from the following more general statement for links.
\begin{theorem}
\label{qpm1}
Let $(B,\pi)$ be an open book supporting $(Y,\xi)$, and $S\subset \Sigma = \pi^{-1}(0)$ a subsurface whose oriented boundary is the union of links $K_1$ and $K_2$. If $K_1$ is boundary compatible,
there exists an $\mathbb{F}[U]$-module homomorphism
\[
f_S: HFK^{=}(-Y, K_1)\to HFK^{=}(-Y, -K_2)
\]
preserving the transverse invariant, i.e. $f_S \big{(}t^=(K_1)\big{)} = t^=(-K_2)$. Moreover, we have a relationship between the $\widehat{t}$-invariants, if $\widehat{t}(-K_2)\ne 0$ then $\widehat{t}(K_1)\ne 0$. 
\end{theorem}

We first consider the case that $S$ is a pair of pants.
\begin{lemma}
\label{pants}
Let $(B,\pi)$ be an open book supporting $(Y,\xi)$, and $S\subset \Sigma = \pi^{-1}(0)$ a pair of pants whose oriented boundary is the union of two links $K_1,K_2$. 
Let $K\subset \Sigma\smallsetminus S$ be a third link such that $K\cup K_1$ is boundary compatible.
Then, there exists a homomorphism
\[
f_S: HFK^=(-Y,K\cup K_1)\to HFK^=(-Y,K\cup -K_2)
\]
of $\mathbb{F}[U]$-modules preserving the transverse invariant, i.e. $f\big{(}t^=(K\cup K_1)\big{)} = t^=(K\cup -K_2)$. Moreover, if $\widehat{t}(K\cup -K_2)\ne 0$ then $\widehat{t}(K\cup K_1)\ne 0$. 
\end{lemma}

\begin{proof}
Because of the orientation conditions, note that $K\cup -K_2$ is boundary compatible as well, so in light of Theorem \ref{uniquetrans} both links are naturally transverse.

We claim that there exists a pointed monodromy $\psi$ having transverse closure $K\cup -K_2$, and a right-hand half twist $\sigma_\gamma$ such that $\psi\circ \sigma_\gamma$ has transverse closure $K\cup K_1$. 
Assuming this claim for the moment, we construct the map $f_S$.

The monodromy $\sigma_\gamma$ has transverse closure $U_n\subset (\#^m (S^1\times S^2),\xi_{std})$.
By Corollary \ref{comult=}, there exists a comultiplication map
\[
\mu^=: HFK^= (-Y,K\cup K_1)\to HFK^= (-Y, K\cup -K_2)\otimes HFK^= (-\#^m (S^1\times S^2), U_n)
\]
sending \[t^= (K\cup K_1) = t^=(\psi\circ \sigma_\gamma) \to t^= (\psi)\otimes t^= (\sigma_\gamma) = t^= (K\cup -K_2)\otimes t^= (U_n) .\]
Composing this map with $\mathds{1}\otimes G^=$, we obtain the desired map $f_S$.
Also, because $\widehat{t}(U_n)\ne 0$ by Lemma \ref{lemqpm}, Theorem \ref{thm:comult} gives
\[
\widehat{t}(K\cup -K_2)\ne 0 \implies \widehat{t}(K\cup K_1)\ne 0.
\]

There are two cases, corresponding to whether $K_1$ is connected or not.

First suppose that $K_1$ is connected.
Let $\{\delta_i\}\cup \alpha$ denote a set of auxiliary arcs for $K\cup K_1$, so that $\alpha$ is an auxiliary arc for $K_1$. The arc $\alpha$ guides an isotopy of $K_1$, as in the discussion following Definition \ref{boundary}. Carrying the surface $S$ along with this isotopy, we see a disk $D$ in a collar neighborhood of $K_1\subset S$ which is fixed by the open book monodromy $\phi$. 
By construction, the monodromy
\[
\pi_\alpha\circ \prod_i \pi_{\delta_i} \circ \phi
\]
has transverse closure $K\cup K_1$.
By positively Markov stabilizing along an arc $\gamma$ and conjugating, we obtain a pointed monodromy
\[
\prod_i \pi_{\delta_i} \circ \phi\circ \sigma_\gamma \circ \pi_\alpha
\]
having transverse closure $K\cup K_1$ - see the left side of Figure \ref{fig:split}.
Consider the push maps $\pi_{\alpha_1}$ and $\pi_{\alpha_2}$ pictured in the center of Figure \ref{fig:split}.

\begin{figure}[h]
\def\svgwidth{400pt}
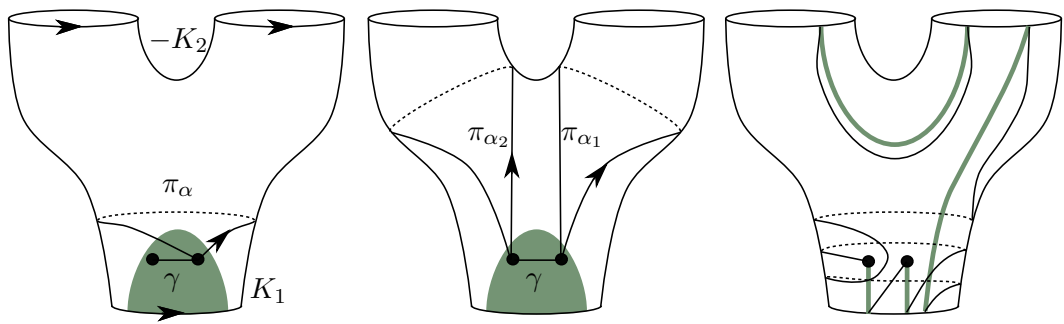
\caption{The pair of pants $S$. Left and center, the intersection of $S$ with a collar neighborhood of $\partial \Sigma $ is shaded. On the right, the arcs $\{\beta_i\}$ are lightly shaded, their images appear in black.}
\label{fig:split}
\end{figure}

The monodromies $\sigma_\gamma \circ \pi_\alpha$ and $\pi_{\alpha_2}\circ \sigma_\gamma \circ \pi_{\alpha_1}$ are supported in $S\smallsetminus \{p_1,p_2\}$, we claim they are isotopic. Consider the set of properly embedded arcs $\{\beta_1,\beta_2,\beta_3,\beta_4\}\subset S\smallsetminus \{p_1,p_2\}$, shown on the right side of Figure \ref{fig:split}, which cut the surface into a disk having no marked points. The two pointed monodromies act identically on these acts, establishing the claim.

Applying the established relation and conjugating, we see that
\[
\prod_i \pi_{\delta_i} \circ \phi\circ \sigma_\gamma \circ \pi_\alpha \sim \Big{(}\pi_{\alpha_1}\circ \prod_i \pi_{\delta_i} \circ \phi\circ \pi_{\alpha_2}\Big{)}\circ \sigma_\gamma =: \psi \circ \sigma_\gamma
\]
has transverse closure $K\cup K_1$. By construction, $\psi$ has transverse closure $K\cup -K_2$. 

We turn to the case that $K_1$ has two components. Let $\{\delta_i\}\cup \{\alpha_1,\alpha_2\}$ denote a set of auxiliary arcs for $K\cup K_1$, so that $\{\alpha_1,\alpha_2\}$ is a set of auxiliary arcs for $K_1$. The auxiliary arcs for $K_1$ guide an isotopy of $S$ so that there are two disks in $S$ fixed by the open book monodromy $\phi$, one in each component of the collar neighborhood of $K_1\subset S$.
By construction, the monodromy
\[
\pi_{\alpha_1}\circ \prod_i \pi_{\delta_i} \circ \phi\circ \pi_{\alpha_2} = \Big{(}\pi_{\alpha_1}\circ \prod_i \pi_{\delta_i} \circ \phi\circ \pi_{\alpha_2}\circ \sigma_{\gamma}^{-1}\Big{)}\circ \sigma_\gamma =: \psi \circ \sigma_\gamma
\]
has transverse closure $K\cup K_1$, where $\gamma$ is the arc pictured in Figure \ref{fig:merge}.

\begin{figure}[h]
\def\svgwidth{400pt}
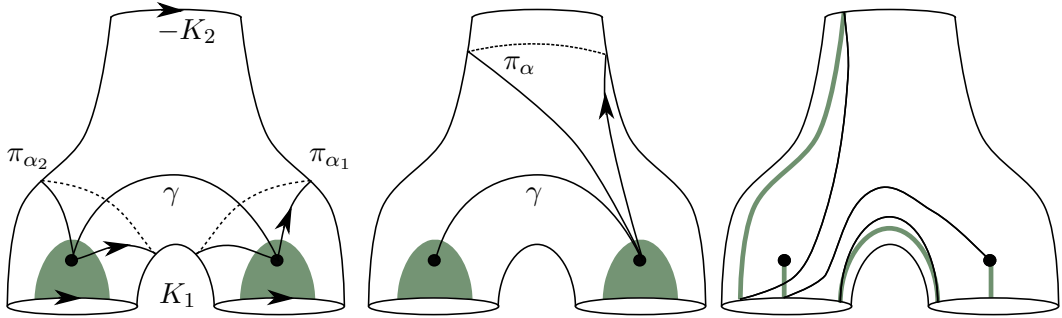
\caption{The pair of pants $S$. Left and center, the intersection of $S$ with a collar neighborhood of $\partial \Sigma $ is shaded. On the right, the arcs $\{\beta_i\}$ are lightly shaded, their images appear in black.}
\label{fig:merge}
\end{figure}

Conjugating $\psi$, we obtain
\[
\psi = \pi_{\alpha_1}\circ \prod_i \pi_{\delta_i} \circ \phi\circ \pi_{\alpha_2}\circ \sigma_{\gamma}^{-1} \sim \prod_i \pi_{\delta_i} \circ \phi\circ \pi_{\alpha_2}\circ \sigma_{\gamma}^{-1}\circ \pi_{\alpha_1}.
\]
Consider the push-map $\pi_\alpha$ pictured in the center  Figure \ref{fig:merge}. The monodromies $\pi_{\alpha_2}\circ \sigma_{\gamma}^{-1}\circ \pi_{\alpha_1}$ and $\sigma_\gamma \circ \pi_\alpha$ are both supported in $S$, we claim they are isotopic. Consider the set of properly embedded arcs $\{\beta_1,\beta_2,\beta_3,\beta_4\}\subset S\smallsetminus \{p_1,p_2\}$, pictured on the right side of Figure \ref{fig:merge}, which cut the surface into a disk having no marked points. The images of $\beta$-arcs under both monodromies are isotopic, the claim follows. 

By applying the above relation, along with a sequence of conjugations and a Markov destabilization, we see that
\[
\psi\sim \pi_{\alpha_1}^{-1} \circ \psi\circ \pi_{\alpha_1} = \prod_i \pi_{\delta_i} \circ \phi\circ \sigma_\gamma \circ \pi_\alpha \sim \pi_\alpha \circ \prod_i \pi_{\delta_i} \circ \phi\circ \sigma_\gamma \xrightarrow{\text{destabilize}} \pi_\alpha \circ \prod_i \pi_{\delta_i} \circ \phi
\]
has transverse closure $K\cup -K_2$.

\end{proof}

\begin{proof}[Proof of Theorem \ref{qpm1}]

We choose a Morse function $f:S\to [0,1]$ satisfying
\begin{itemize}
\item $f^{-1}(0) = K_1$,
\item $f^{-1}(1) = K_2$, 
\item the critical points of $f$ are index all one, and 
\item $f$ takes on distinct values $0<c_1<\dots<c_n<1$ at the critical points.
\end{itemize}

The function $f$ specifies a decomposition of $S$ into $n$-subsurfaces $\{S_i\}_{i=1}^{n}$ in the usual way.
Set $c_0 = -c_1$ and $c_{n+1} = 2-c_n$, and let $S_i := f^{-1}([\frac{c_i+c_{i-1}}{2},\frac{c_{i+1}+c_{i}}{2}])$. 
Each $S_i$ is a disjoint union a pair of pants along with some number of cylinders.
For each $1\le i\le n$, denote by $L_i$ the oriented component $f^{-1} (\frac{c_i+c_{i-1}}{2})$ of $\partial S_i$.
Note that $K_1=L_1$ and set $L_{n+1}:= -K_2$.

By Lemma \ref{pants}, each subsurface $S_i$ gives rise to a map
\[
f_{S_i}: HFK^{=}(-Y, L_i)\to HFK^{=}(-Y, L_{i+1})
\]
preserving the transverse invariant, $t^=(L_i)\xrightarrow{f_{S_i}} t^= (L_{i+1})$. The map desired is the composition
\[
f_S:= f_{S_n}\circ\dots\circ f_{S_2}\circ f_{S_1}:  HFK^{=}(-Y, K_1)\to HFK^{=}(-Y, -K_2).
\]

Lemma \ref{pants} also gives that
\[
\widehat{t}(-K_2) = \widehat{t}(L_{n+1})\ne 0 \implies \widehat{t}(L_n)\ne 0 \implies \dots \implies \widehat{t}(L_1) = \widehat{t}(K_1)\ne 0.
\]
\end{proof}

\begin{corollary}
\label{corsqp}
Suppose $K\subset (Y,\xi)$ is a strongly quasipositive link in a contact manifold having non-vanishing contact invariant, $c(\xi)\ne 0$. 
Let $S$ be a quasipositive Seifert surface for $K$, realized as a $\pi_1$-injective subsurface of a page of an open book $(B,\pi)$ supporting $(Y,\xi)$; this realizes $K$ as a boundary compatible link in the page and hence as a transverse link by Theorem \ref{uniquetransapprox}. Then $K$ has non-vanishing transverse invariant, $\widehat{t}(K)\ne 0$.

\end{corollary}

\begin{proof}
Let $S'$ denote the surface obtained by removing a disk $D$ from the interior of $S$. Let $U$ denote the unknotted boundary of $D$.
The natural transverse representative of $-U$ coming from Theorem \ref{uniquetransapprox} is $U_1$. 
By Theorem \ref{qpm1}, it suffices to check that $\widehat{t}(U_1)\ne 0$. This follows from $c(\xi)\ne 0$ via the isomorphism of Lemma \ref{meridians}.



\end{proof}

\begin{corollary}
\label{qpmap}
Let $(B,\pi)$ be an open book supporting $(Y,\xi)$, and $S\subset \Sigma=\pi^{-1}(0)$ a subsurface. Suppose that the oriented boundary of $S$ is the union of two non-isolating links $L_1,L_2$. 
Then, there are homomorphisms 
\begin{align*}
f_{1}: HFK^=(-Y,L_1)\to HFK^=(-Y,-L_2)\\
f_{2}: HFK^=(-Y,L_2)\to HFK^=(-Y,-L_1)
\end{align*}
preserving the Legendrian invariant, i.e. $f_{1}\big{(}t^{=}(L_1)\big{)} = t^{=}(-L_2)$ and $f_{2}\big{(}t^=(L_2)\big{)}= t^=(-L_1)$. 
Moreover, $\widehat{t}(-L_2)\ne 0 \implies \widehat{t}(L_1)\ne 0$ and $\widehat{t}(-L_1)\ne 0 \implies \widehat{t}(L_2)\ne 0$.
\end{corollary}
\begin{proof}
A non-isolating link is boundary compatible, with either choice of orientation.
We may view $S$ as a cobordism from $L_1$ to $-L_2$ or from $L_2$ to $-L_1$. 
Applying Theorems \ref{uniquetransapprox} and \ref{qpm1} gives the result.


\end{proof}

\section{Naturality under contact +1 surgery}
\label{sec:contactsurg}



Ozsv\'ath and Stipsicz \cite{contactsurgery} established naturality of the invariant $\mathcal{L}$ of a Legendrian knot $L$ under contact +1 surgery along some other Legendrian knot $K$.
In light of Theorem \ref{equiv}, the following is a natural generalization to the setting of Legendrian links.

If $K$ is Legendrian realized by a non-isolating curve in the page of an open book supporting $(Y,\xi)$, then contact +1 surgery along $K$ corresponds to composing the open book monodromy with a left-hand Dehn twist $\tau_K ^{-1}$ about $K$ \cite{gayconc}.

\begin{theorem}
\label{contactsurgery}
Suppose $L\subset (Y,\xi)$ is an $n$ component Legendrian link, and $K\subset Y\smallsetminus L$ is a Legendrian knot. Let $L_K$ denote the image of $L$ in the contact +1 surgery $(Y_K,\xi_K)$. There exist homomorphisms of $\mathbb{F}$, $\mathbb{F}[U]$, and $\mathbb{F}[U_1,\dots,U_n]$-modules
\[
HFK^\circ (-Y,L)\to HFK^\circ(-Y_K,L_K)
\]
sending $t^\circ (L)$ to $t^\circ (L_K)$, where $\circ\in\{\wedge,=,-\}$.
\end{theorem}
\begin{proof}

We may choose a contact cell decomposition of $(Y,\xi)$ whose 1-skeleton contains $K\cup L$. A ribbon of the 1-skeleton is a page $\Sigma = \pi^{-1}(0)$ of an open book $(B,\pi)$ supporting $(Y,\xi)$ having monodromy $\phi$. $K\cup L$ naturally sits in $\Sigma$ as a non-isolating link. 
Choosing a set of auxiliary arcs $\{\alpha_i\}$ for $L$ we obtain a pointed monodromy 
\[
\phi\circ \prod_i \pi_{\alpha_i} =\Big{(}\phi\circ \prod_i \pi_{\alpha_i} \circ \tau_{K}^{-1}\Big{)} \circ \tau_K =: \psi \circ \tau_K
\]
whose closure is transversely isotopic to $L_+$.

The closure of $\psi$ is transversely isotopic to the positive transverse push-off of $L_K\subset Y_K$, and the closure of $\tau_K$ is transversely isotopic to $U_n\subset \#^m (S^1\times S^2)$.
By Theorems \ref{thm:comultiplication}, \ref{thm:comult}\footnote{since the three pointed monodromies are pure braids, the comultiplication map for the totally blocked theory $\widetilde{HFK}$  is one for the hat theory $\widehat{HFK}$, see Remark \ref{hatcomult}.}, and Corollary \ref{comult=} , there exist comultiplication maps
\[
\mu^\circ : HFK^\circ(-Y,L)\to HFK^\circ (-Y_K,L_K)\otimes HFK^\circ (-\#^m (S^1\times S^2),U_n)\\
\]
sending
\[
t^\circ (L)= t^\circ (\psi\circ \tau_K)  \xrightarrow{\mu^\circ} t^\circ (\psi)\otimes t^\circ(\tau_K) =t^\circ(L_K)\otimes t^\circ(U_n).
\]
Composing with $\mathds{1}\otimes G^\circ$ gives the desired map; $G^\circ$ is the map defined after Lemma \ref{lemqpm}.
\end{proof}

Sahamie  \cite{sahamie} established a functorial connection between the invariant $\mathcal{L}$ of a Legendrian knot and the contact invariant of the resulting contact +1 surgery. The following theorem is a generalization to the setting of Legendrian links:

\begin{theorem}
Suppose $L\subset (Y,\xi)$ is an $n$ component Legendrian link, and $K$ a Legendrian knot. 
Let $L_K$ denote the image of $L$ in the contact +1 surgery $(Y_K,\xi_K)$.
There exist homomorphisms of $\mathbb{F}[U]$ and $\mathbb{F}$-modules
\[
HFK^\circ(-Y,K\cup L)\to HFK^\circ (-Y_K,L_K)\]
\[\widehat{HFK}(-Y,K)\to \widehat{HF}(-Y_K)
\]
sending $t^\circ (K\cup L)$ to $t^\circ (L_K)$, and $\widehat{t}(K)$ to $c(\xi_K)$, respectively; here $\circ \in \{\wedge,=\}$.
\end{theorem}

\begin{proof}
We construct the first map.
We may choose a contact cell decomposition of $(Y,\xi)$ whose 1-skeleton contains $K\cup L$. A ribbon of the 1-skeleton is a page $\Sigma = \pi^{-1}(0)$ of an open book $(B,\pi)$ supporting $(Y,\xi)$ having monodromy $\phi$. $K\cup L$ naturally sits in $\Sigma$ as a non-isolating link. Let $\alpha\cup \{\alpha_i\}$ denote a set of auxiliary arcs for $K\cup L$, with $\alpha$ an auxiliary arc for $K$. 

We obtain a pointed monodromy
$
\phi \circ \prod_i \pi_{\alpha_i} \circ \pi_\alpha
$
having closure $(K\cup L)_+$.
Recall the construction of $\pi_\alpha$. The arc $\alpha$ guides an isotopy of $K$ within $\Sigma$ to $K'$ which contains a sub-arc $\beta$ lying in a collar neighborhood of $\partial \Sigma$ fixed by the monodromy $\phi$. The map $\pi_\alpha=\pi_{K'}$ is a push-map of a marked point $p\in \alpha\cap\beta$ about $K'$. Let $A$ be a small neighborhood of $K'$, and $\partial A = \delta_1 \cup \delta_2$. Suppose $K'$ is isotopic to $\delta_2$ as an oriented curve, then, by definition of the push-map we have $\pi_\alpha = \tau_{\delta_1}^{-1}\circ \tau_{\delta_2}$.

Thus, $(K\cup L)_+$ is transversely isotopic to the closure of 
\[
\phi \circ \prod_i \pi_{\alpha_i} \circ \pi_\alpha = \Big{(} \phi \circ \prod_i \pi_{\alpha_i} \circ \tau_{\delta_1}^{-1} \Big{)} \circ \tau_{\delta_2} =: \psi \circ \tau_{\delta_2}.
\]
The transverse closure of $\psi$ is $(L_K)_+\cup U_1\subset (Y_K,\xi_K)$, and the transverse closure of $\tau_{\delta_2}$ is $U_n\subset (\#^m (S^1\times S^2),\xi_{std})$. By Corollary \ref{comult=} and  Theorem \ref{thm:comult}, there exist comultiplication maps
\[
\mu^\circ : HFK^\circ(-Y,K\cup L)\to HFK^\circ (-Y_K,L_K\cup U_1)\otimes HFK^\circ (-\#^m (S^1\times S^2),U_n)
\]
sending\[
t^\circ(K\cup L) = t^\circ (\psi\circ\tau_{\delta_2}) \xrightarrow{\mu^\circ} t^\circ (\psi)\otimes t^\circ(\tau_{\delta_2})= t^\circ (L_K\cup U_1)\otimes t^\circ (U_n).
\]
Composing $\mu^\circ$ with $h^\circ\otimes G^\circ$ gives the desired map, where 
$h^\circ$ and $G^\circ$ are the maps defined in Lemma \ref{birth} and after Lemma \ref{lemqpm}, respectively.

The construction of the second map is very similar. The monodromy
\[
\phi \circ \pi_\alpha = \Big{(}\phi\circ\tau_{\delta_1}^{-1}\Big{)}\circ \tau_{\delta_2} =:\psi '\circ \tau_{\delta_2}
\]
has closure transversely isotopic to $K$. The closure of $\psi'$ is $U_1\subset (Y_K,\xi_K)$, and the closure\footnote{here, we are viewing $\tau_{\delta_2}$ as an element of the mapping class group of $\Sigma$ equipped with a single marked point. In the construction of the first map we viewed $\tau_{\delta_2}$ as an element of $Mod(\Sigma\smallsetminus \{p_1,\dots,p_n\},\partial \Sigma)$.} of $\tau_{\delta_2}$ is $U_1\subset (\#^m (S^1\times S^2),\xi_{std})$. By Theorem \ref{thm:comult}, there exists a comultiplication map
\[
\widehat{\mu} : \widehat{HFK}(-Y,K)\to \widehat{HFK} (-Y_K,U_1)\otimes \widehat{HFK} (-\#^m (S^1\times S^2),U_1)\\
\]sending
\[\widehat{t}(K) = \widehat{t}(\psi ' \circ \tau_{\delta_2}) \xrightarrow{\widehat{\mu}} \widehat{t} (\psi)\otimes \widehat{t} (\tau_{\delta_2}) = \widehat{t}(U_1)\otimes \widehat{t}(U_1).
\]
Composing $\widehat{\mu}$ with $\widehat{g}\otimes \widehat{G}$ gives the desired map, where $\widehat{g}$ and $\widehat{G}$ are the maps defined in Lemma \ref{meridians} and after Lemma \ref{lemqpm}, respectively.

\end{proof}

\section{Lagrangian link cobordisms}
\label{sec:lagr}

Let $L_1,L_2\subset (Y,\xi)$ be Legendrian links. 
We say that $L_2$ is obtained from $L_1$ via a \emph{birth} (\emph{pinch move}) if there exists a Darboux ball $B\subset (Y,\xi)$ such that the front projections of $L_1\cap B$ and $L_2\cap B$ differ as pictured on the top (respectively bottom) of Figure \ref{fig:decomp}.

\begin{figure}[h]
\def\svgwidth{100pt}
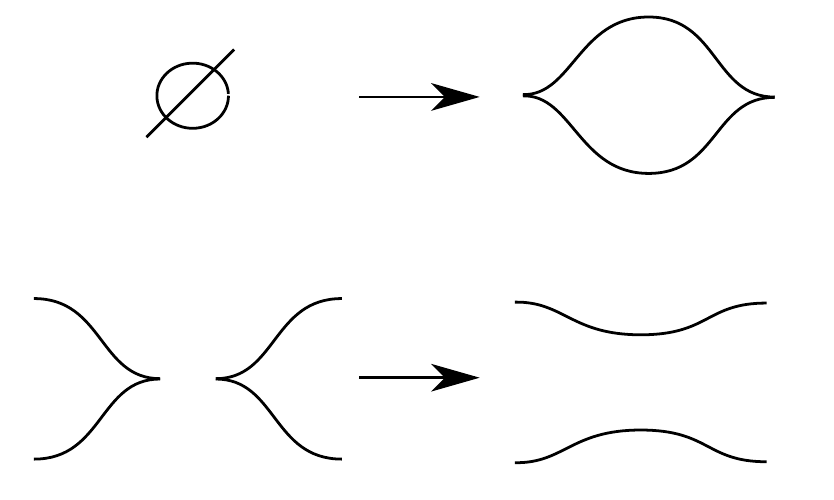
\caption{Birth and pinch moves.}
\label{fig:decomp}
\end{figure}

\begin{definition}

Let $\alpha$ be a 1-form on $Y$ whose kernel specifies $\xi$. The \emph{symplectization} of $(Y,\xi)$ is the symplectic manifold $Symp(Y,\xi):=(Y\times \mathbb{R}, d(e^t \alpha))$. 
A link cobordism $S\subset Symp(Y,\xi)$ from $L_1$ to $L_2$ is \emph{decomposable Lagrangian cobordism} if it may be constructed by stacking elementary cobordisms associated to births and pinches.
\end{definition}

Contact +1 surgery along a Legendrian $K$ admits an inverse operation, called \emph{Legendrian surgery} \cite{DGSF}, which 4-dimensionally corresponds to a Weinstein 2-handle attachment along $K$\cite{ehweinstein}. If $(B,\pi)$ is an open book supporting $(Y,\xi)$ with monodromy $\phi$, then performing Legendrian surgery along a non-isolating curve $K$ in the page $\pi^{-1}(0)$ corresponds to composing $\phi$ with a right-handed Dehn twist $\tau_K$.

Weinstein 1-handle attachments also admit an interpretation via open books. There are two cases, either the 1-handle connects two disjoint components $(Y_1,\xi_1)$ and $(Y_2,\xi_2)$ and corresponds to a contact connected sum of these, or the 1-handle is attached to a single component and corresponds to taking a contact connected sum with $(S^1\times S^2,\xi_{std})$. In either case, one attaches a 2-dimensional 1-handle to the page of some open book and extends the monodromy trivially \cite{gayconc}. 

\begin{definition}
Let $L\subset (Y,\xi)$ be a Legendrian, and consider the Lagrangian $L\times [0,1]$ in a compact piece $(Y\times [0,1],d(e^t \alpha))$ of the symplectization. Attaching Weinstein handles along $(Y\smallsetminus L)\times \{1\}$, we obtain a Weinstein cobordism $(W,\omega):(Y,\xi)\to (Y',\xi')$. The surface $S:=L\times [0,1]$ considered as a subset of $(W,\omega)$ is a \emph{cylindrical Lagrangian cobordism}.
A \emph{decomposable Lagrangian cobordism in a Weinstein cobordism} is one which can be built up by stacking decomposable and cylindrical Lagrangian cobordisms.
\end{definition}

To make sense of naturality of the Legendrian invariant under decomposable Lagrangian cobordisms inside Weinstein cobordisms, we extend the definition of the Legendrian invariant.
Let $(Y,\xi)$ be a disjoint union of connected contact manifolds $\bigcup_{i=1}^n (Y_i,\xi_i)$ and $L = \bigcup _{i=1}^m L_i$ be Legendrian link with $L_i\subset (Y_i,\xi_i)$; $m\le n$. Then for $\circ\in\{\wedge,=,-\}$ we define
\[
t^\circ (L):= \bigotimes_{i=1}^m t^\circ (L_i) \otimes \bigotimes_{i=m+1}^n c(\xi_i)\subset \bigotimes_{i=1}^m HFK^\circ (-Y_i,L_i)\otimes \bigotimes_{i=m+1}^n \widehat{HF}(-Y_i) =: HFK^\circ(-Y,L).
\]
This definition is natural from the perspective of sutured Floer homology, where the Legendrian invariant is interpreted as the contact invariant of the complement of a standard neighborhood of the Legendrian link \cite{stipV}, see also \cite{zarev}.


\begin{theorem}
\label{decomp}
Let $(W,\omega): (Y_1,\xi_1)\to (Y_2,\xi_2)$ be a Weinstein cobordism. A decomposable Lagrangian cobordism $S\subset W$ between Legendrian links $L_i\subset (Y_i,\xi_i)$ induces an $\mathbb{F}[U]$-module homomorphism 
\[
f_S: HFK^{=}(-Y_2,L_2)\to HFK^{=}(-Y_1,L_1)
\]
which preserves the Legendrian invariant, $t^=(L_2)\xrightarrow{f_S} t^=(L_1)$. Moreover, we have a relationship between the $\widehat{t}$-invariants, if $\widehat{t}(L_1)\ne 0$ then $\widehat{t}(L_2)\ne 0$. 
\end{theorem}

\begin{proof}
It suffices to establish the result for $S$ a cylindrical Lagrangian in a Weinstein cobordism built with a single Weinstein handle, and for $S$ an elementary cobordism associated to a single birth or pinch. The case of a Weinstein 0-handle is trivial, as this corresponds to taking the disjoint union with a tight 3-sphere, which has non-vanishing contact invariant. 
A Weinstein 1-handle corresponds to a contact connected sum between pieces of $Y_1$ which may or may not contain parts of the link $L_1$; Lemma \ref{connectedsum} takes care of all cases. 

A Weinstein 2-handle corresponds to performing Legendrian surgery along a knot $K\subset Y_1$ disjoint from $L_1$. If the connected component of $Y_1$ containing $K$ is not disjoint from $L$ we apply Theorem \ref{contactsurgery}. Otherwise, we appeal to the naturality of the contact invariant under Legendrian surgery \cite{contactclass}.\footnote{this also follows from comultiplicativity of the contact invariant \cite{comultcontact}.}

Consider the case of a birth. $L_2$ is the disjoint union of $L_1$ with a $tb=-1$ Legendrian approximation of the $sl=-1$ transverse unknot $U_1$ constructed in Section \ref{nat}. 
If the birth occurs in a connected component of $Y_1$ disjoint from $L_1$, we apply Lemma \ref{meridians}, otherwise we apply Lemma \ref{birth}.


Now, suppose that $S$ is an elementary cobordism associated to a pinch move. We will show that there exists a surface sitting in a page of an open book supporting $(Y,\xi)$, whose oriented boundary is the union of two non-isolating links, one giving rise to the Legendrian $S_-(L_2)$ and the other $-L_1$. Given such a surface, Theorem \ref{qpm1} provides a map
\[
HFK^= (-Y, L_2)\to HFK^= (-Y,L_1)
\]
sending \[t^= (L_2) = t^= (S_-(L_2)_+)\to t^= ((L_1)_+) = t^= (L_1)\]
and also gives that $\widehat{t}(L_1)\ne 0 \implies \widehat{t}(L_2)\ne 0$.

For a pinch move there are two cases, corresponding to the two possible orientations on $L_1$.
Let $B$ denote the Darboux ball where the pinch move occurs.
First suppose that $L_1\cap B$ is oriented as on the left side of Figure \ref{fig:pinch}.  
We construct a Legendrian graph $\Gamma$ by adding a single edge to $L_1$ within $B$, pictured in the center of Figure \ref{fig:pinch}.

\begin{figure}[h]
\def\svgwidth{230pt}
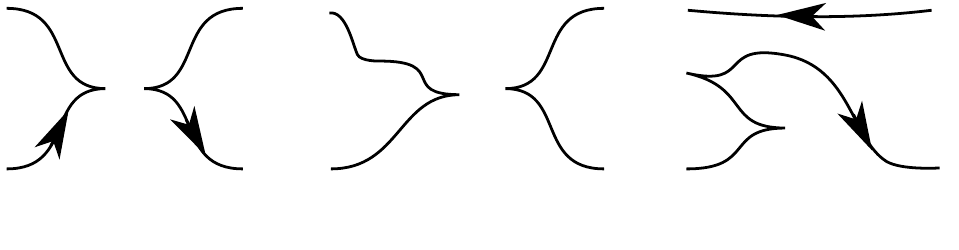
\caption{}
\label{fig:pinch}
\end{figure}

Following Avdek \cite{avdek}, we draw the front projection of a ribbon $R(\Gamma)\cap B$ on the left-hand side of Figure \ref{fig:ribbon}.

\begin{figure}[h]
\def\svgwidth{255pt}
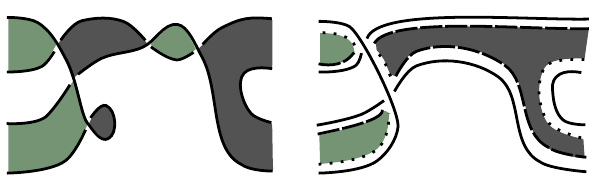
\caption{On the left, a front projection for the ribbon $R(\Gamma)$ is pictured; the positive side of the ribbon is shaded darker. On the right, we have our subsurface $\Sigma$, with oriented boundary $K_1$, which is dotted, and $K_2$, which is dashed.}
\label{fig:ribbon}
\end{figure}

The ribbon $R(\Gamma)$ consists of a disjoint union of some number of quasi-positive annuli and a single pair of pants. 
Let $\Sigma \subset R(\Gamma)$ be a subsurface obtained by removing a collar neighborhood of $\partial R(\Gamma)$ from $R(\Gamma)$.
The oriented boundary of $\Sigma$ can be written as the union of two non-isolating links $K_1\cup K_2\subset R(\Gamma)$, as on the right side of Figure \ref{fig:ribbon}. By construction, $K_1$ and $K_2$ give rise to the Legendrians $-L_1$ and $S_-(L_2)$, respectively. By extending $\Gamma$ to the 1-skeleton of a contact cell-decomposition $G$, we realize $\Sigma\subset R(\Gamma)\subset R(G)$ as a sub-surface of a page $R(G)$ of an open book supporting $(Y,\xi)$. 
The construction for the other orientation is identical, after reflecting the figures about the vertical axis.



\end{proof}


\nocite{comultcontact, comultgrid, weinstein}


\nocite{holknots}

\newpage
\thispagestyle{empty}
{\small
\markboth{References}{References}
\bibliographystyle{myalpha}
\bibliography{mybib}{}
}

\end{document}